\documentclass[11pt, a4paper]{amsart}

\usepackage{amsfonts,amsmath,amssymb, amscd}

\usepackage[all]{xy}

\newtheorem{theorem}{Theorem}[section]
\newtheorem{lemma}[theorem]{Lemma}
\newtheorem{prop}[theorem]{Proposition}
\newtheorem{corollary}[theorem]{Corollary}
\newtheorem{lemmadef}[theorem]{Lemma-Definition}

\theoremstyle{definition}
\newtheorem{definition}[theorem]{Definition}
\newtheorem{rem}[theorem]{Remark}

\newcommand{\lcosmash}{\mathop{\raisebox{0.2ex}{\makebox[0.92em][l]{${\scriptstyle>\mathrel{\mkern-4mu}\blacktriangleleft}$}}}}
\newcommand{\rcosmash}{\mathop{\raisebox{0.2ex}{\makebox[0.92em][l]{${\scriptstyle\blacktriangleright\mathrel{\mkern-4mu}<}$}}}}

\newcommand{\ot}{\otimes}
\newcommand{\mr}{\mathrm}
\newcommand{\ms}{\mathsf}
\newcommand{\op}{\operatorname}
\newcommand{\G}{\mathbb{G}}
\renewcommand{\H}{\mathbb{H}}
\newcommand{\Lie}{\mathrm{Lie}}
\newcommand{\A}{\mathbb{A}}
\newcommand{\B}{\mathbb{B}}
\newcommand{\C}{\mathbb{C}}
\newcommand{\D}{\mathbb{D}}
\newcommand{\X}{\mathbb{X}}
\newcommand{\Y}{\mathbb{Y}}

\renewcommand{\S}{\operatorname{Spec}}
\renewcommand{\SS}{\operatorname{SSpec}}
\newcommand{\z}{\mathsf{Z}}
\newcommand{\zz}{\wedge(\mathsf{Z})}
\newcommand{\wG}{\mathsf{W}_{\mathbb{G}}}
\newcommand{\wwG}{\wedge(\mathsf{W}_{\mathbb{G}})}
\newcommand{\wH}{\mathsf{W}_{\mathbb{H}}}
\newcommand{\wwH}{\wedge(\mathsf{W}_{\mathbb{H}})}
\newcommand{\SM}{\mathsf{SMod}}
\newcommand{\os}{\overset}
\newcommand{\tto}{\longrightarrow}
\renewcommand{\O}{\mathcal{O}}
\newcommand{\gr}{\operatorname{gr}}

\numberwithin{equation}{section}

\title[Quotients $G/H$ in supersymmetry]
{Geometric construction of quotients $G/H$\\ in supersymmetry}

\author[A.~Masuoka]{Akira Masuoka}
\address{Akira Masuoka,
Institute of Mathematics, 
University of Tsukuba, 
Ibaraki 305-8571, Japan}
\email{akira@math.tsukuba.ac.jp}

\author[Y.~Takahashi]{Yuta Takahashi}
\address{Yuta Takahashi,
Graduate School of Pure and Applied Sciences, 
University of Tsukuba, 
Ibaraki 305-8571, Japan}
\email{y-takahashi@math.tsukuba.ac.jp}

\begin{document}

\begin{abstract}
It was proved by the first-named author and Zubkov \cite{MZ1} that given an affine algebraic 
supergroup $\G$ and a closed sub-supergroup $\H$ over an arbitrary field of characteristic $\ne 2$, 
the faisceau $\G \tilde{/} \H$ (in the fppf topology)
is a superscheme, and is, therefore, the quotient superscheme $\G/\H$, which has some desirable properties, in fact. 
We reprove this, by constructing directly the latter superscheme $\G/\H$. 
Our proof describes explicitly the structure sheaf of $\G/\H$, and reveals some new
geometric features of the quotient, that include one which was desired by Brundan \cite{B}, and
is shown in general, here for the first time. 
\end{abstract}

\maketitle

\noindent
{\sc Key Words:}
affine algebraic supergroup, 
Hopf superalgebra, 
superscheme,
faisceau

\medskip
\noindent
{\sc Mathematics Subject Classification (2000):}
14L15, 
14M30, 
16T05  

\section{Introduction}\label{sec:introduction}

Throughout in this paper we work over a fixed, arbitrary field $\Bbbk$ of characteristic $\ne 2$. 
Algebras, Hopf algebras, schemes and so on, together with their super-analogues, all are those over $\Bbbk$. 
The unadorned $\otimes$ means the tensor product $\otimes_{\Bbbk}$ over $\Bbbk$. 

\subsection{Quotients $G/H$}\label{intro1}
Given a group $G$ and a subgroup $H \subset G$, one has the set $G/H$ of cosets. 
This elementary fact which one learns at the first Algebra Course immediately turns into
a difficult question in the context of schemes, in which $G$ is an affine algebraic group scheme and $H$
is a closed subgroup scheme of $G$.  However, we already know the answer that there exists uniquely a scheme $G/H$
which fits in with the natural co-equalizer diagram $G \times H \rightrightarrows G \to G/H$ of schemes, and
which has desirable properties, such as being Noetherian;
see \cite[Part I, Sections 5.6--5.7]{J}, for example.  

It is easy to pose the same question in the generalized, super situation. But it was brought to 
our interest not long ago, by \cite{B} (2006).
In this article
J.~Brundan  listed up some properties that supersymmetric
quotients should have, and showed some general results, assuming the existence of the quotient. 
Moreover, he proved that there exists such a quotient $\G/\H$ with the properties 
for a special algebraic supergroup $\G=Q(n)$ and its parabolic sub-supergroups $\H=P_{\gamma}$, and applied his general
results to $Q(n)\supset P_{\gamma}$, 
producing beautiful results on representations of $Q(n)$. 
Later, the first named author and A.~Zubkov \cite{MZ1} (2011) proved the existence of quotients in general,
showing their properties which, however, do not include one from Brundan's list.
In this paper we reconstruct the quotient more directly,
describing its structure sheaf explicitly, and show its properties which include all (reasonable ones) in Brundan's list; the
list will be examined finally in Section \ref{subsec:Brundan}. 
Consequently, Brundan's general results turn applicable to a wider class of affine algebraic supergroups.

\subsection{Supersymmetry}\label{intro2}
The word ``super" is a synonym of ``graded by the order-$2$-group $\mathbb{Z}/(2)=\{\, 0, 1\, \}$"; the $0$ (resp., the $1$)
in $\mathbb{Z}/(2)$ 
is called \emph{even} (resp., \emph{odd}).
A \emph{super-vector space} is thus a vector space $V$ which is $\mathbb{Z}/(2)$-graded so as $V=V_0\oplus V_1$; 
$V$ is said to be \emph{purely even} (resp., \emph{purely odd}) if $V=V_0$ (resp., if $V=V_1$). 
The super-vector spaces $V, W,\dots$ all together form a symmetric tensor category with respect to the natural tensor product
$V \ot W$, the unit object $\Bbbk$ and the supersymmetry
\[ 
c=c_{V,W} : V \ot W \os{\simeq}{\tto} W \ot V,\ \, c(v\ot w)=(-1)^{|v||w|}w \ot v, 
\]
where $v$ and $w$ are supposed to be homogeneous elements of degree $|v|$, $|w|$, respectively. 
Ordinary objects, such as algebra, commutative algebra, Hopf algebra, which are defined in the 
symmetric tensor category of vector spaces, equipped with the trivial symmetry $V\ot W \os{\simeq}{\tto} W \ot V$,
$v\ot w \mapsto w \ot v$, are generalized by \emph{super-objects} defined in the symmetric tensor category
of super-vector spaces; the objects are called with ``super" attached, so as superalgebra, super-commutative
superalgebra, Hopf superalgebra. Ordinary objects are precisely purely even super-objects. 

In what follows, superalgebras $\A$ (and Hopf superalgebras as well) are assumed to be super-commutative, unless otherwise stated;
the assumption means that $\A_0$ is a central subalgebra of $\A$, and we have $ab=-ba$ for all $a, b \in \A_1$. 
Accordingly, (Hopf) algebras are assumed to be commutative. 

\subsection{Geometrical vs.~functorial viewpoints}\label{intro3}
The article \cite{MZ1} showed that the circumstance around schemes is directly generalized to the super situation, as follows. 
The notion of superschemes is defined in two ways, 
from geometrical viewpoint and from functorial viewpoint; the notion from the latter will be called
a \emph{functorial superscheme} in this paper. 
Roughly speaking, a \emph{superscheme} is a topological space, equipped with a structure sheaf of superalgebras,
which is covered by some affine open sub-superschemes; an affine superscheme, $\S \A$, is uniquely given 
by a superalgebra, say $\A$, so that the underlying topological space is the
the spectrum $\S(\A_0)$ of the algebra $\A_0$, and the superalgebra $\O_{\S \A}(\S(\A_0))$ of global sections
is $\A$.
A \emph{functorial superscheme} is a set-valued functor
defined on the category of superalgebras, which is, roughly speaking, the union of some affine open sub-functors; a functorial affine
superscheme, $\op{Sp} \A$, is a representable functor, which is thus uniquely represented by a superalgebra, say $\A$. 
The Comparison Theorem \cite[Theorem 5.14]{MZ1} states that $\S \A \mapsto \op{Sp} \A$ naturally extends to an
equivalence from the category of superschemes to the category of functorial superschemes. 
An advantage of the functorial viewpoint is in that
the latter category is included in the tractable category of faisceaux; a \emph{faisceau} is a functor which behaves
like a sheaf with respect to the so-called fppf-coverings of superalgebras. 

Group-objects in the category of (functorial) superschemes are called {supergroup schemes}. 
But we treat only affine supergroup schemes in this paper.
In addition, when we discuss affine (super)group schemes (not affine (super)schemes), we 
omit the word ``scheme", and say \emph{affine (super)groups}, following the widely known custom of Jantzen \cite{J},

\subsection{Main result and consequences}\label{intro4}
Let $\G=\S \C$ be an affine algebraic supergroup, and
$\H =\S \D$ a closed sub-supergroup.
Thus, $\C$ is a finitely generated
Hopf superalgebra, and $\D$ is a quotient Hopf superalgebra of $\C$. 
(Warning: the symbol $\C$ is thus used to denote Hopf superalgebras, and it will never be used  to denote
the field of complex numbers in this paper.)
It is easy to construct the quotient $\G\tilde{/}\H$ in the category of faisceaux. One principle 
is that if the faisceau $\G\tilde{/}\H$ happens to be a functorial superscheme, we have
the quotient $\G/\H$ in the category of superschemes by the Comparison Theorem. 
In fact, the article \cite{MZ1} referred to in Section \ref{intro1} has proved that the assumption
is satisfied, to obtain the conclusion. But we only depend on the principle in the restricted situation that
the quotient is affine.
Being more on the geometrical side, we construct the superscheme $\G/\H$ directly, as follows.

One sees that 
$\G$ (resp., $\H$) includes an affine algebraic group $G =\S C$ (resp., $H=\S D$) as the largest 
purely even closed sub-supergroup. We remark that $\G$ and $G$ (resp., $\H$ and $H$) has the same underlying
topological space, so that $|G|=|\G|\supset |\H|=|H|$, whence $G \supset H$. 
Let $\pi : G \to G/H$ be the quotient morphism; to this, known results can apply. 
Choose arbitrarily an affine open subset $\emptyset \ne U \subset |G/H|$. Then $\pi^{-1}(U)$ 
is an $H$-stable affine open subscheme of $G$ such that $\pi^{-1}(U)/H = U$.  
Note that $\pi^{-1}(U)$ is an open subset of $\G$, as well. The key of ours is
to construct an $\H$-equivariant embedding of some right $\H$-equivariant affine superscheme onto $\pi^{-1}(U)$ in $\G$. 
Such an embedding has the form $\S(\omega)$, where
$\omega : \C \to \A$ is a map of right $\D$-super-comodule superalgebras; the question is, therefore,
to find an appropriate right $\D$-super-comodule superalgebra $\A$ together with $\omega$ such as above. 
Indeed, Hopf-algebraic techniques enable us to find out very useful ones; 
see Proposition \ref{prop:open_embed} and Corollary \ref{cor:quotient_local}. 
The result is that the $\pi^{-1}(U)$ in $\G$
is an $\H$-stable affine open sub-superscheme of $\G$, such that $\pi^{-1}(U)/\H$ exists, and is an affine superscheme.
Our main theorem, Theorem \ref{mainthm}, shows that the thus obtained affine superschemes, when $U$ ranges over all affine open subsets of $|G/H|$,
are uniquely glued into a superscheme with the underlying topological space $|G/H|$, and the resulting superscheme
is indeed the quotient $\G/\H$: the underlying topological space $|\G/\H|$ is thus the same as $|G/H|$. 
The proof will give a new description of the structure sheaf $\O_{\G/\H}$ (Remark \ref{rem:about_mainthm}): $\O_{\G/\H}$ is 
\emph{locally} isomorphic to the sheaf
\[ \wedge_{\O_{G/H}}(\pi_*\O_G\, \square_D\, \z),  \]
where $\pi_*\O_G\, \square_D\, \z$ is a locally free $\O_{G/H}$-module sheaf. It is this property that was failed to be shown by \cite{MZ1};
see Section \ref{intro1}. 

It does happen that the sheaves $\O_{\G/\H}$ and $\wedge_{\O_{G/H}}(\pi_*\O_G\, \square_D\, \z)$ are not \emph{globally} isomorphic; see
Remark \ref{rem:known_for_split}.
On the other hand, 
Proposition \ref{prop:split} gives some sufficient conditions for the two sheaves to be globally isomorphic. 
The new description above shows that $\G/\H$ has desirable properties; see Proposition \ref{prop:affinity2}. 
They include the remarkable one: 
an open subset of $|\G/\H|\, (=|G/H|)$
is affine in $\G/\H$ if and only if it is affine in $G/H$.  

The results over-viewed above are contained in Section \ref{sec:G/H}. 
The preceding two sections are devoted to preliminaries.
Section \ref{sec:preliminaries1} summarizes basic facts on super-(co)algebras and superschemes; they 
include the Comparison Theorem, Theorem \ref{thm:comparison}, referred to above.
Section \ref{sec:preliminaries2} is devoted mostly to reproducing
necessary, known results on affine supergroups and Hopf superalgebras.


\section{Superalgebras and superschemes}\label{sec:preliminaries1}

This preliminary section summarizes basic facts on super-(co)algebras and on superschemes 
in Sections \ref{subsec:super_vs_non-super}--\ref{subsec:Noetherian_superalgebra}
and in Sections \ref{subsec:superscheme}--\ref{subsec:functor}, respectively.


\subsection{Super vs. non-super situations}\label{subsec:super_vs_non-super}
Super-(co)algebras are regarded as ordinary (co)algebras, with the 
$\mathbb{Z}/(2)$-grading forgotten. 
A right, say, supermodule $M$ over a superalgebra $\B$ is
(faithfully) flat as an ordinary right $\B$-module if and only if the functor $M\otimes_{\B}$ defined
on the category of left $\B$-supermodules is (faithfully) exact
\cite[Lemma 5.1 (1)]{M1}.
Similarly, a (left or right) super-comodule over a super-coalgebra $\C$ is injective 
(or equivalently, coflat) as an ordinary $\C$-comodule if and only if it is so in the category of
$\C$-super-comodules. 
If the equivalent conditions are satisfied we say simply
that the object in question is (\emph{faithfully}) \emph{flat} or \emph{injective}. 

Recall that given a left $\C$-super-comodule $L=(L,\ \lambda_L : L \to \C \ot L)$ and a right $\C$-super-comodule 
$M=(M,\ \rho_M : M \to M\ot \C)$, the 
\emph{co-tensor product} $M\, \square_{\C}\, L$ is the super-vector space defined as the equalizer
\[ M\, \square_{\C}\, L\to M\ot L \rightrightarrows M\ot \C \ot L \]
of $\op{id}_M \ot \lambda_L$ and $\rho_M \ot \op{id}_L$.  
The functor $M \square_{\C}$ (resp., $\square_{\C}\, L$)
defined on the category of left (resp., right) $\C$-(super-)comodules is left exact. If it is exact,
then $M$ (resp., $L$) is said to be \emph{coflat}. 
The condition is equivalent to the $\C$-(super-)comodule being injective, as noted above; see \cite[Proposition A.2.1]{T1}.


\subsection{Superalgebras}\label{subsec:superalgebra}
Recall that all superalgebras are assumed to be super-commutative. 
Given a superalgebra $\B$, left $\B$-supermodules and right $\B$-supermodules are identified by a canonical
category-isomorphism \cite[Lemma 5.2 (2)]{M1}.
It follows that a $\B$-superalgebra $\A$ is faithfully flat as a left $\B$-(super)module 
if and only if it is so as a right $\B$-(super)module \cite[Lemma 5.3 (2)]{M1}. 
In this case we say that $\A$ is \emph{faithfully flat} over $\B$, or $\B \to \A$ is \emph{faithfully flat}. 
We say that $\A$ is \emph{fppf} 
(fid\`{e}lement plat et de pr\'{e}sentation finie) over $\B$ if it is
faithfully flat and finitely presented. Recall that $\A$ is said to be \emph{finitely presented} over $\B$
if it has the form $\B[\underline{\mathsf{X}}|\underline{\mathsf{Y}}]/I$, where  
$\B[\underline{\mathsf{X}}|\underline{\mathsf{Y}}]=
\B[\mathsf{X}_1,\dots,\mathsf{X}_r|\mathsf{Y}_1,\dots,\mathsf{Y}_s]$ is a polynomial superalgebra 
in finitely many even variables $\underline{\mathsf{X}}=(\mathsf{X}_i)_i$ and odd variables 
$\underline{\mathsf{Y}}=(\mathsf{Y}_i)_i$, 
and $I$ is a finitely generated super-ideal. 


\subsection{Graded superalgebras}\label{subsec:graded_superalgebra}
Let $\A=\A_0\oplus \A_1$ be a superalgebra. The super-ideal $I_{\A}=(\A_1)$ 
generated by the odd component $\A_1$ is the smallest super-deal such that the
quotient
\begin{equation}\label{eq:assoc_algebra}
A := \A/I_{\A} \, (=\A_0/\A_1^2) 
\end{equation}
is an ordinary (commutative) algebra. This last algebra is said to be \emph{associated with}
the original superalgebra, denoted by the corresponding normal capital letter. 
The descending chain $\A \supset I_{\A}\supset I_{\A}^2 \supset \dots$ of super-ideals constructs 
the \emph{graded superalgebra} 
\begin{equation}\label{eq:grA}
\gr \A:= \bigoplus_{n\ge 0}I_{\A}^n/I_{\A}^{n+1}
\end{equation}
\emph{associated with} $\A$.
By a \emph{graded superalgebra} we mean an algebra graded by $\mathbb{N}=\{ 0,1,2,\dots\}$ which,
regarded as $\mathbb{Z}/(2)$-graded by mod-$2$ reduction, is a super-commutative superalgebra. 
Note that the $A$-module $(\gr \A)(1)=I_\A/I_\A^2\, (=\A_1/\A_1^3)$ is purely odd, and 
the embedding $I_\A/I_\A^2 \hookrightarrow \gr \A$ induces a surjection
of graded superalgebras 
\[ \wedge_A(I_\A/I_\A^2) \to \gr \A \]
from the exterior $A$-algebra on the $A$-module $I_\A/I_\A^2$. 


\subsection{Noetherian superalgebras}\label{subsec:Noetherian_superalgebra}
Retain the notation as above.
We say that $\A$ is \emph{Noetherian} if its super-ideals satisfy the ACC.
The condition is easily seen to be equivalent to each of the following:
\begin{itemize}
\item[(i)]
The commutative algebra $\A_0$ is Noetherian, and the $\A_0$-algebra $\A$ is generated by finitely many odd elements; 
\item[(ii)] 
$\A_0$ is Noetherian, and the $\A_0$-module $\A_1$ is finitely generated;
\item[(iii)]
$A$ is Noetherian, and the $A$-module $I_\A/I_\A^2$ is finitely generated;
\item[(iv)]
The superalgebra $\wedge_A(I_\A/I_\A^2)$ is Noetherian; 
\item[(v)]
The superalgebra $\gr \A$ is Noetherian. 
\end{itemize}
See \cite[Section A.1]{MZ2}. Given a Noetherian superalgebra $\B$, a finitely generated $\B$-superalgebra
is finitely presented over $\B$, and is Noetherian. 

Let $\A$ and $\B$ be Noetherian superalgebras. Then $\gr \A$ and $\gr \B$ are finitely graded
in the sense that $(\gr \A)(n)=0=(\gr \B)(n)$ for $n \gg 0$. It follows that
a superalgebra map $f : \A \to \B$ is surjective/injective if and only if the associated
graded superalgebra map $\gr f :\gr \A \to \gr \B$ is so. 


\subsection{Superschemes}\label{subsec:superscheme}
A \emph{super-ringed space} (over $\Bbbk$) is a topological space equipped with a
sheaf of superalgebras (over $\Bbbk$) on it. It is said to be \emph{local}
if the stalk at every point is local; see below. 

Let $\A=\A_0\oplus \A_1$ be a superalgebra with associated algebra $A=\A/(\A_1)$. 
The affine superscheme $\S \A$ associated with $\A$ is a local-super-ringed space. 
Its underlying topological space is the spectrum $\S(\A_0)$ of the algebra $\A_0$; 
it is naturally identified with
the spectrum $\S A$ of $A$, since $A=\A_0/\A_1^2$ and $\A_1^2\subset \sqrt{0}$.
Note that every super-ideal $\mathbb{P}$ of $\A$ such that
$\A/\mathbb{P}$ is an integral domain uniquely has the form $\mathbb{P}=P\oplus \A_1$
with $P\in \S(\A_0)$. Similarly, every proper super-ideal of $\A$ that is maximal with respect to
inclusion uniquely has the form $P \oplus \A_1$, where $P\subset \A_0$ is a maximal ideal. 
We say that $\A$ is \emph{local} if it has a unique maximal super-ideal, or equivalently, if $\A_0$ is local.
The localization $S^{-1}\A$ by a multiplicative set $S \subset \A_0$ is 
the base extension $\A\otimes_{\A_0}S^{-1}\A_0$ of the $\A_0$-algebra $\A$ along the localization
$\A_0 \to S^{-1}\A_0$. If $S=\A_0\setminus P$ with $P\in \S(\A_0)$ (resp., if $S=\{1,x,x^2,\dots \}$ with
$x \in \A_0$), then $S^{-1}\A$ is denoted by $\A_P$ (resp., $\A_x$), as usual. 
Note that $\A_P$ is local. 
The structure sheaf $\O_{\S \A}$ of $\S \A$ is the unique 
sheaf of superalgebras that assigns $\A_x$ to every principal open set $D(x)=\{ P\mid x \notin P \}$. 
The stalk $\O_{\S \A,P}$ at $P$ is $\A_P$. 
In \cite{MZ1}, $\S \A$ is alternatively denoted $\SS A$. 

A \emph{superscheme} (over $\Bbbk$) is a local-super-ringed space (over $\Bbbk$) 
which is locally isomorphic to some affine superscheme. 
The superschemes form a full subcategory of the category of local-super-ringed spaces.
A morphism $f : \X \to \Y$ of the latter category is required to be such that the
induced superalgebra map $f^*_P : \O_{\Y,f(P)}\to \O_{\X,P}$ between the stalks
is \emph{local}, that is, $f^*_P$ sends the maximal super-ideal of $\O_{\Y,f(P)}$ into that of $\O_{\X,P}$.  
Basic notions for schemes and their morphisms, such as
\emph{algebraic/Noetherian scheme}, \emph{open/closed embedding}, \emph{affine/faithfully flat/finitely-presented morphism},
and relevant basic results are generalized to our super context in the obvious manner. 

A superscheme $\X$ is said to be \emph{smooth} at point $P$ of the underlying topological space $|\X|$, if
the stalk $\O_{\X,P}$ at $P$ is smooth as a superalgebra; this means that a superalgebra 
surjection onto $\O_{\X,P}$ splits whenever its kernel is a nilpotent super-ideal. 
A superscheme is said to be \emph{smooth} if it is smooth at every point. 
Theorem A.2 of \cite{MZ2} gives some characterizations for a Noetherian affine superscheme to be smooth.


\subsection{The associated graded superscheme}\label{subsec:graded_superscheme}
Let $\X$ be a superscheme with structure sheaf $\O_{\X}$. 
Given a non-empty affine open subset $U \subset |\X|$, we have the affine superscheme
\[ \Y_{U}:=\S(\gr \O_{\X}(U)) \]
given by the graded superalgebra $\gr \O_{\X}(U)$ associated with the superalgebra $\O_{\X}(U)$. 
The underlying topological space $|\Y_U|$ is naturally identified with that space of $\S(\O_\X(U)/(\O_\X(U)_1))$,
and hence with $U$. 

\begin{lemmadef}\label{lemdef:graded_superscheme}
The affine superschemes $\Y_U$, where $U$ ranges over non-empty affine open subsets of $|\X|$, are uniquely glued into
a superscheme with the underlying topological space $|\X|$. 

\emph{We denote the resulting superscheme by $\gr \X$, and call it} the graded superscheme associated with $\X$. 
\end{lemmadef}
\begin{proof}
Let $U \supset U'$ be affine open in $\X$. Suppose $U=\S \A$, $U'=\S \A'$, and that
$U \supset U'$ arise from a superalgebra map $i : \A \to \A'$. Let $A=\A/(\A_1)$ and $A'=\A'/(\A'_1)$ be the associated
algebras. 
Choose $P \in \S(\A'_0)$ arbitrarily, and 
set $Q=i^{-1}(P) \, (\in \S(\A_0))$. Then the map of stalks $i_P: \A_Q \to \A'_P$ at $P$ is an isomorphism. 
Note that the operation $\gr$ commutes with localization, $\gr(i_P)= (\gr i)_P$, and so
\[ (\gr i)_P: \gr(\A)_Q \to \gr(\A')_P \]
is an isomorphism. Here we may suppose that $P \in \S A'$, $Q \in \S A$, and that the relevant localizations
$\gr(\A')_P$ and $\gr(\A)_Q$ are by those, since the $\A'_0$-algebra $\gr(\A')_P$ and $\A_0$-algebra
$\gr(\A)_Q$ are, in fact, an $A'$-algebra and an $A$-algebra, respectively. The result remains unchanged
if we replace $P$ and $Q$ with the corresponding primes in $\S(\gr(\A')_0)$ and in $\S(\gr(\A)_0)$, respectively. 
Indeed, the localizations are then unchanged since one may ignore the deference by nilpotent elements for localizing elements. 

To prove the assertion it suffices to prove that the structure sheaf $\O_{\Y_U}$ of $\Y_U$, pull-backed to $U'$, coincides with
$\O_{\Y_{U'}}$. But this follows from the result just proven. 
\end{proof}

The structure sheaf of $\gr \X$ is a sheaf of graded superalgebras. 
One sees that a superscheme $\X$ is Noetherian if and only if $\gr \X$ is. 
The argument of the last proof, concentrating in degree zero, shows the following. 

\begin{lemmadef}\label{lemdef:associated_scheme}
The affine schemes
\[ \S \big(\O_\X(U)/(\O_\X(U)_1)\big) \]
where $U$ ranges over non-empty affine open subsets of $|\X|$, are uniquely glued into
a scheme with the underlying topological space $|\X|$. 

\emph{We call the resulting scheme} the scheme associated with $\X$. 
\end{lemmadef}

Let $\X$ be a superscheme. 
We say that $\X$ is \emph{split}, if there exists a scheme $X$
with the same underlying topological space $|X|=|\X|$ as that space of $\X$, together with 
a locally free $\O_X$-module sheaf $\mathcal{M}$, such that 
\[ \O_\X\simeq \wedge_{\O_X}(\mathcal{M}), \]
where $\mathcal{M}$ is supposed to be purely odd. 
One sees that $X$ is necessarily the scheme associated with $\X$. 

If $\X$ is split, then $\X=\gr \X$. 
We see from \cite[Theorem A.2]{MZ2} that $\gr \X$ is split if $\X$ is Noetherian and smooth; 
see the proof of Proposition \ref{prop:grG/H}.  
In fact, the assumption can be weakened to $\X$ being Noetherian and regular
(in the sense of Schmitt \cite{Schmitt}); see \cite[Definition A.1]{MZ2}.


\subsection{Functorial viewpoint}\label{subsec:functor}
Let $\ms{SAlg}$ and $\ms{Set}$
denote the categories of superalgebras (over $\Bbbk$) and of sets, respectively. 
A $\Bbbk$-\emph{functor} is a functor $\mathfrak{F}: \ms{SAlg}\to \ms{Set}$. 
Let
\[
\ms{Func}= \ms{Set}^{\ms{SAlg}}
\]
denote the category of
$\Bbbk$-functors and natural transformations. 
A $\Bbbk$-functor $\mathfrak{F}$
is called a \emph{faisceau} (resp., \emph{faisceau dur}), if it
preserves finite direct products and if it turns every equalizer diagram of superalgebras
\[ \B \to \A \rightrightarrows \A \otimes_\B \A \]
that naturally arises from an fppf (resp., faithfully flat) map $\B \to \A$
(the paired arrows indicate
$a \mapsto a\ot 1,\ 1 \ot a$) into an equalizer diagram of sets
\[ \mathfrak{F}(\B) \to \mathfrak{F}(\A) \rightrightarrows \mathfrak{F}(\A \otimes_\B \A).  \]

Given $\A \in \ms{SAlg}$, we let
\[ \op{Sp}\, \A = \ms{SAlg}(\A, -) : \ms{SAlg}\to \ms{Set},\ \mathbb{T}\mapsto \ms{SAlg}(\A, \mathbb{T}) \]
denote the $\Bbbk$-functor represented by $\A$; this is alternatively denoted $\op{SSp}\A$ in \cite{MZ1}. 
Such a representable $\Bbbk$-functor is called 
a \emph{functorial affine superscheme}; see the last paragraph of this section. 
The definitions of \emph{open sub-functors} and of \emph{local functors} given by \cite[Part I, 1.7, 1.8]{J} in the ordinary
situation are directly generalized to the super situation; see \cite[Sect. 3]{MZ1}. 
A \emph{functorial superscheme} is a local $\Bbbk$-functor $F$ which is the
\emph{union}, $\mathfrak{F}=\bigcup_{i}\mathfrak{G}_i$, of some affine open sub-functors $\mathfrak{G}_i$ in the sense that
$\mathfrak{F}(K)=\bigcup_{i}\mathfrak{G}_i(K)$ for every field $K$ including $\Bbbk$. 

The category $\ms{Func}$ includes full subcategories in the relation:
\[
\begin{pmatrix}\text{functorial}\\ \text{affine superschemes}\end{pmatrix} 
\subset
\begin{pmatrix}\text{functorial}\\ \text{superschemes}\end{pmatrix}\subset
(\text{faisceaux dur})\subset (\text{faisceaux}).
\]
See \cite[Proposition 5.15]{MZ1}. 
The cited article \cite{MZ1} puts emphasis on the functorial viewpoint, while we do more
on the geometrical viewpoint;
the article calls (dur) $\Bbbk$-sheaves, (affine) superschemes and geometric superschemes
what we call faisceaux (dur), functorial (affine) superschemes and superschemes, respectively. 

Given a superscheme $\X$, the $\Bbbk$-functor
\[ \X^{\diamond} : \ms{SAlg}\to \ms{Set},\ \X^{\diamond}(\mathbb{T})=\op{Mor}(\S \mathbb{T}, \X), \]
where $\op{Mor}$ denotes the set of the morphisms of superschemes, 
is proved to be a functorial superscheme (see \cite[Lemma 5.2]{MZ1}),
and is called the functorial 
superscheme \emph{represented by} $\X$. We say that $\X$ \emph{represents} $\X^{\diamond}$.
For example, the affine superscheme $\S \A$ represents $\op{Sp} \A$; see \cite[Lemma 4.1]{MZ1}. 

Here we reproduce from \cite[Theorem 5.14]{MZ1} the Comparison Theorem:

\begin{theorem}\label{thm:comparison}
$\X \mapsto \X^{\diamond}$ gives rise to a category-equivalence from the category 
of superschemes to the category of functorial superschemes. 
\end{theorem}

For an explicit quasi-inverse see \cite[Proposition 5.12, Lemma 5.13]{MZ1}. 

An affine superscheme $\X=\S \A$ and 
the assigned, functorial affine superscheme $\X^{\diamond}=\op{Sp} \A$ 
are both controlled by the superalgebra $\A$, and may not be distinguished in many situations.
We will call the latter as well, an \emph{affine superscheme}, omitting the word ``functorial", 
as usual. Even when one has to distinguish them,  
which is meant will be clear from the context or the notation.


\section{Affine supergroups and Hopf superalgebras}\label{sec:preliminaries2}

This section is devoted again to preliminaries, which include reproducing  two fundamental theorems on
affine supergroups and Hopf superalgebras.


\subsection{Affine supergroups}\label{subsec:supergroup}
One sees just as in the non-super situation that
the two categories treated in the last theorem have finite direct products (and more generally, fiber products). 
Therefore, both of them have group objects,
which we call \emph{supergroup schemes} and \emph{functorial supergroup schemes}, respectively.
The proved category-equivalence induces a category-equivalence between those group objects. 
But in what follows, we will discuss only affine supergroup schemes; they are precisely 
affine superschemes 
\begin{equation}\label{eq:supergroup}
\S \D,\quad \op{Sp} \D 
\end{equation}
equipped with group structure, which uniquely arises from a Hopf-superalgebra structure on $\D$. 
We call the two of \eqref{eq:supergroup} both an \emph{affine supergroup},
omitting the word ``scheme", following the custom of Jantzen \cite{J}. It is called an
\emph{affine algebraic supergroup} if the Hopf superalgebra $\D$ is finitely generated as a superalgebra.

Let $\G=\S \D$ be an affine supergroup.  A right $\D$-super-comodule is the same as a left $\G$-super-module. 
Given such a super-comodule $M=(M, \rho_M)$, the super-vector space $M^{co\D}$ of all $\D$-\emph{coinvariants} in $M$
is defined by
\begin{equation}\label{eq:coinvariants_in_M}
M^{co\D}=\{\, m\in M\mid \rho_M(m)=m\ot 1\, \}.
\end{equation}
This is identified with the co-tensor product $M\, \square_{\D}\, \Bbbk$, where $\Bbbk$ is the trivial, purely even left $\D$-super-comodule, 
and also with the super-vector space $M^{\G}$ of all $\G$-\emph{invariants} in $M$. 


\subsection{Affinity criteria}\label{subsec:affinity_criteria}
Let $\X=\S \A$ be an affine superscheme, and let $\G=\S \D$ be an affine supergroup.  
Suppose that $\G$ acts on $\X$ from the right. 
This means that there is given a morphism of (functorial) superschemes
$\X \times \G \to \X$,  called an \emph{action} by $\G$ on $\X$,
which satisfies the familiar associativity and unit-property. Such an action
arises uniquely from a right $\D$-super-comodule superalgebra structure
\[ \rho_{\A} : \A \to \A \otimes \D \]
on $\A$; it is by definition a superalgebra map with which $\A$ is a right $\D$-super-comodule. 
Let 
\begin{equation}\label{eq:coinvariants}
\B=\A^{co \D}\, (=\{\, b \in \A \mid \rho_{\A}(b)=b\otimes 1\, \}). 
\end{equation}
This is a sub-superalgebra of $\A$.
Let $\SM_{\B}$ denote
the category of right $\B$-supermodules. A super-vector space $M$ equipped with a right $\A$-supermodule
structure and a right $\D$-super-comodule structure 
\[ M \ot \A \to M,\ m\ot a \mapsto ma;\quad \rho_M : M \to M \otimes \D \]
is called a $(\D,\A)$-\emph{Hopf supermodule} (see \cite[p.454]{T2}) if it satisfies
\[ \rho_M(ma)=\rho_M(m)\, \rho_{\A}(a),\quad m \in M,\ a \in \A. \]
Let $\SM^{\D}_{\A}$ denote the category of $(\D,\A)$-Hopf
supermodules; the morphisms are $\A$-supermodule and $\D$-super-comodule maps.  Obviously, $\A \in \SM^{\D}_{\A}$. 
The categories $\SM_{\B}$ and $\SM_{\A}^{\D}$ are both $\Bbbk$-linear abelian. 
Given an object $N \in \SM_{\B}$, 
the right $\A$-supermodule $N\ot_{\B}\, \A$, equipped with
the right $\D$-super-comodule structure 
$\op{id}\ot_{\B}\, \rho_{\A} : N \ot_{\B}\A \to N \ot_{\B} (\A \ot \D)=(N \ot_{\B}\, \A) \ot \D$, 
turns into a $(\D,\A)$-Hopf supermodule. This construction gives rise to a $\Bbbk$-linear functor
\begin{equation}\label{eq:Hopf_module_functor}
\SM_{\B}\to \SM^{\D}_{\A},\ N\mapsto N\ot_{\B}\, \A, 
\end{equation}
which is left adjoint to
\begin{equation}\label{eq:coinv_functor}
\SM^{\D}_{\A}\to \SM_{\B},\ M \mapsto M^{co\D}.
\end{equation}

The following theorem, which is reproduced from \cite{MZ1}, 
is a super-analogue of U.~Oberst's Satz A of \cite{O};
see Remark \ref{rem:Schneider} (1) below. 
Some notion and notation used here will be explained soon below.

\begin{theorem}[\text{\cite[Theorem 7.1]{MZ1}}]\label{thm:superOberst}
Retain the situation as above.
\begin{itemize}
\item[(1)] The following are equivalent: 
\begin{itemize}
\item[(i)]
The action by $\G$ on $\X$ is free, and the faisceau dur $\X \tilde{\tilde{/}} \G$ is 
an affine superscheme;
\item[(ii)]
\begin{itemize}
\item[(a)]
$\A$ is injective as a right $\D$-super-comodule, and
\item[(b)]
the map 
\begin{equation}\label{eq:alpha}
\alpha=\alpha_{\A} : \A \otimes \A \to \A \otimes \D,\ \alpha(a \otimes a') = a\hspace{0.4mm} \rho_{\A}(a')
\end{equation}
is a surjection;
\end{itemize}
\item[(iii)]
\begin{itemize}
\item[(a)]
$\A$ is faithfully flat over $\B$, and
\item[(b)]
the map 
\begin{equation}\label{eq:beta}
\beta : \A \otimes_{\B} \A \to \A \otimes \D,\ \beta(a \otimes_{\B} a') = a\hspace{0.4mm} \rho_{\A}(a')
\end{equation}
induced from the map $\alpha$ in (ii) is a bijection.
\end{itemize}
\item[(iv)]
The functors \eqref{eq:Hopf_module_functor} and
\eqref{eq:coinv_functor} are (necessarily, mutually quasi-inverse) equivalences.   
\end{itemize}
If these equivalent conditions are satisfied, then $\X \tilde{\tilde{/}} \G=\op{Sp} \B$.
\item[(2)]
Suppose that $\G$ is algebraic, or in other words, $\D$ is finitely generated. Suppose that $\A$ is Noetherian. 
If the equivalent conditions above are satisfied, 
then $\X \tilde{\tilde{/}} \G$ is Noetherian (or equivalently, $\B$ is Noetherian), and it coincides with the faisceau $\X \tilde{/} \G$.
\end{itemize}
\end{theorem}

In the situation above we say that the action $\X \times \G \to \X$ is 
\emph{free} (see Condition (i) above), if for every $\mathbb{T}\in \ms{SAlg}$,
the action $\X(\mathbb{T}) \times \G(\mathbb{T})\to \X(\mathbb{T})$,\ $(x,g)\mapsto x^g$ is free; 
this last means that $x^g=x$ implies $g=1$,
or equivalently, that 
$\X(\mathbb{T}) \times \G(\mathbb{T})\to \X(\mathbb{T})\times \X(\mathbb{T})$,\ $(x, g)\mapsto (x, x^g)$
is injective. Obviously, the action $\X \times \G \to \X$ is free if Condition (ii)(b) above is satisfied. 

Given a $\Bbbk$-functor $\mathfrak{F} \in \ms{Func}$, there exists uniquely a faisceau $\tilde{\mathfrak{F}}$ 
equipped with a morphism $\mathfrak{F} \to \tilde{\mathfrak{F}}$ in $\ms{Func}$ such that for any faisceau $\mathfrak{G}$, 
the map $\ms{Func}(\tilde{\mathfrak{F}}, \mathfrak{G}) \to \ms{Func}(\mathfrak{F},\mathfrak{G})$ induced by the morphism is a bijection;
see \cite[Proposition 3.6]{MZ1}. We have the faisceau dur $\tilde{\tilde{\mathfrak{F}}}$ with the
analogous universality for faisceaux dur. If $\mathfrak{F}$ preserves finite direct products and 
has the property that if $\mathbb{S}\to \mathbb{T}$ is fppf (resp., faithfully flat), 
$\mathfrak{F}(\mathbb{S}) \to \mathfrak{F}(\mathbb{T})$ is an injection, then the construction of $\tilde{\mathfrak{F}}$ 
(resp., of $\tilde{\tilde{\mathfrak{F}}}$) is quite simple, and we have 
\[ \mathfrak{F}(\mathbb{T})\subset \tilde{\mathfrak{F}}(\mathbb{T}) \subset \tilde{\tilde{\mathfrak{F}}}(\mathbb{T}) \] 
for every $\mathbb{T}\in \ms{SAlg}$;
see \cite[Remark 3.8]{MZ1}. This is the case if $\mathfrak{F}$ is the $\Bbbk$-functor which assigns to
every $\mathbb{T}$, the set $\X(\mathbb{T})/\G(\mathbb{T})$ of $\G(\mathbb{T})$-orbits
in $\X(\mathbb{T})$, provided the action by $\G$ on $\X$ is free. 
In this case $\tilde{\mathfrak{F}}$ (resp., $\tilde{\tilde{\mathfrak{F}}}$) is denoted 
\[ \X\tilde{/}\G\quad (\text{resp.,}\ \X\tilde{\tilde{/}}\G). \]

For our purpose it is enough to work only with free actions.

\begin{definition}[\text{cf. \cite[Definition 8.1.1]{Mo}}]\label{def:Galois}
 If the equivalent conditions (i)--(iv) in Part 1 of Theorem \ref{thm:superOberst} are satisfied,
we say that $\A \supset \B$ is a $\D$-\emph{Galois extension}.
\end{definition}
 
\begin{rem}\label{rem:Schneider}
(1)\
H.-J.~Schneider \cite[Theorem I]{S} proved the Oberst Satz cited above in the non-commutative, purely Hopf-algebraic
situation, in which $\A$, $\B$ and $\D$ as above may be non-commutative
while neither $\G$ nor $\X$ is referred to. 
In \cite{MZ1} the theorem 
reproduced above was derived from Schneider's Theorem, by using the bosonization technique 
developed therein.

(2)\ As is remarked by \cite[Remark 7.2]{MZ1}, $\SM^{\D}_{\A}=(\SM^\D_\A, \ot_\A, \A)$ is a tensor category with respect to
the tensor product $\ot_\A$ over $\A$ and the unit object $\A$.  In particular, $\A$ is an algebra object in it. Moreover,
the functor \eqref{eq:Hopf_module_functor} is a tensor functor
$\SM_{\B}=(\SM_\B, \ot_\B, \B) \to \SM^{\D}_{\A}$.
\end{rem}

\subsection{The quotient superscheme $\X/\G$}\label{subsec:X/G}
Suppose that an affine supergroup $\G$ acts freely on an affine superscheme $\X$ from the right.
The \emph{quotient superscheme} $\X/\G$ 
is a superscheme equipped with a morphism from $\X$, such that
\[ \X \times \G \rightrightarrows \X \to \X/\G \]
is a co-equalizer diagram of superschemes, where the paired arrows indicate the original $\G$-action and the trivial
$\G$-action. If such a superscheme exists it is unique in the obvious sense.  The morphism $\X \to \X/\G$ will not 
be referred to if it is obvious. 

\begin{lemma}\label{lem:X/G}
If the faisceau $\X\tilde{/}\G$ happens to be a functorial superscheme, then
the quotient superscheme $\X/\G$ exists, and it necessarily represents $\X\tilde{/}\G$. 
\end{lemma}
\begin{proof}
As is seen from the construction of $\X\tilde{/}\G$,
we have the co-equalizer diagram $\X \times \G \rightrightarrows \X \to \X\tilde{/}\G$ of faisceaux. 
This is a diagram of functorial superschemes under the assumption above. Now, Theorem \ref{thm:comparison} proves
the lemma. 
\end{proof}


\subsection{Tensor product decomposition of a Hopf superalgebra}\label{subsec:Hopf}
Let $\G =\S \C$ be an affine supergroup. Thus $\C=\C_0\oplus \C_1$ is a Hopf superalgebra. We assume
that $\G$ is an algebraic supergroup, or in other words, $\C$ is finitely generated.
For later use we set this assumption, without which many of what follows, however, are known to be true. 
The coproduct, the counit and the antipode of this or any other Hopf superalgebra will be denoted so as
\begin{equation}\label{eq:Hopf_struc}
\Delta_\C :\C \to \C \ot \C,\quad \varepsilon_\C : \C \to \Bbbk,\quad \mathcal{S}_\C : \C \to \C, 
\end{equation}
respectively. 

The super-ideal $(\C_1)$ generated by $\C_1$ is a Hopf super-ideal of $\C$, so that the algebra
\[ C =\C/(\C_1) \]
associated with $\C$ (see \eqref{eq:assoc_algebra}) is a quotient, ordinary Hopf algebra of $\C$,
which is obviously 
finitely generated. The associated affine algebraic group $G=\S C$ may be regarded as an affine algebraic
supergroup such that $G(\mathbb{T})=\G(\mathbb{T}_0)$,\ $\mathbb{T}\in \ms{SAlg}$.
The underlying topological space $|\G|$ of $\G$ is naturally identified with that space $|G|$ of $G$, and 
we have the closed embedding $\G \supset G$ which is identical on the underlying topological space. 

Let $q : \C \to C$ denote the quotient map. The composite 
\[ \lambda_{\C} : \C \os{\Delta_{\C}}{\longrightarrow} \C \ot \C \os{q \ot \op{id}}{\longrightarrow} C \ot \C, \]
makes $\C$ into a 
left $C$-super-comodule superalgebra; this is equivalent to saying that $\G$ is a left $G$-equivariant
superscheme. 
Note that $\C$ includes $\C_0$ as a $C$-comodule subalgebra. 
Analogously to \eqref{eq:coinvariants}, we let
\begin{equation}\label{eq:left_coinvariants}
{}^{coC}\C=\{\, x \in \C \mid \lambda_{\C}(x)=1\ot x\, \}
\end{equation}
denote the sub-superalgebra of $\C$ consisting of all \emph{left} $C$-\emph{coinvariants} in $\C$. 

\begin{lemma}[\text{\cite[Footnote 5]{M3}}]\label{lem:CcoC}
There exists a left $C$-comodule algebra map 
\[ \xi : C \to \C_0 \]
such that $q \circ \xi = \op{id}_C$. 
It gives rise to an isomorphism,
\begin{equation}\label{eq:CcoC}
C \ot {}^{coC}\C \os{\simeq}{\longrightarrow} \C,\ \, c \otimes x \mapsto \xi(c)x,  
\end{equation}
of left $C$-super-comodule superalgebras.
\end{lemma}

For a proof of the lemma see Remark \ref{rem:tensor_prod_decomp} (2) below. 

Let $\C^+=\op{Ker}(\varepsilon_{\C})$ denote the augmentation super-ideal of $\C$. 
Since $\C$ is finitely generated, it follows that for every $n >0$, $(\C^+)^n$ is co-finite-dimensional,
or $\dim(\C/(\C^+)^n)< \infty$. By definition
\[ T^*_{\varepsilon}(\G)= \C^+/(\C^+)^2 \]
is the \emph{cotangent super-vector space} of $\G$ at the identity element; 
it is finite-dimensional. The odd component $T^*_{\varepsilon}(\G)_1$ of this
$T^*_{\varepsilon}(\G)$ is denoted by
\[ \wG= \C_1/\C_0^+\C_1, \] 
where $\C_0^+=\C_0 \cap \C^+$. The even component $T^*_{\varepsilon}(\G)_0$ is seen to coincide with
the cotangent space $T^*_{\varepsilon}(G)$ of $G$ at the identity element. 
The right adjoint action
\[ \G(\mathbb{T})\times G(\mathbb{T}) \to \G(\mathbb{T}),\ (h,g)\mapsto g^{-1}hg;\quad \mathbb{T} \in \ms{SAlg} \]
induces a left $G$-supermodule structure, or equivalently, a right $C$-super-comodule structure on $\C$,
which in turn induces such a structure on $T^*_{\varepsilon}(\G)$, and hence on $\wG$ by restriction. 
The resulting $G$-action
on $\wG$ is called the \emph{left co-adjoint action}. The right $G$-action on $\wG$ analogously induced 
from the left $G$-adjoint action $(g, h)\mapsto ghg^{-1}$ on $\G$ is called the \emph{right co-adjoint action}. 

The dual super-vector spaces $(\C/(\C^+)^n)^*$ of $\C/(\C^+)^n$, $n > 0$, amount to
\emph{the hyper-superalgebra of} $\G$
\[ \op{hy}(\G)=\bigcup_{n>0}(\C /(\C^+)^n)^* \]
in $\C^*$. This is a super-cocommutative Hopf superalgebra with trivial coradical,
which is not necessarily super-commutative; see \cite[Sections 2.5, 4.3]{M2}, for example.
The Lie superalgebra
\[ P(\op{hy}(\G))= (\C^+/(\C^+)^2)^* \]
consisting of all primitive elements in $\op{hy}(\G)$ is
\emph{the Lie superalgebra} $\Lie(\G)$ \emph{of} $\G$.
Its odd component is dual to $\wG$, or in notation,
\begin{equation}\label{eq:Lie1_W}
\Lie(\G)_1=\wG^*.
\end{equation}
The left (or right) $G$-action on $\op{Lie}(\G)_1$ dual to the right (or left) co-adjoint $G$-action
on $\wG$ is called the \emph{adjoint action}. 
We remark that the even component $\Lie(\G)_0$ of $\Lie(\G)$ coincides with the Lie algebra
$\Lie(G)$ of the affine algebraic group $G$.  

The exterior algebra $\wwG$ on $\wG$ has the natural Hopf-superalgebra structure with every element
in $\wG$ being primitive; thus, the counit is such that $\varepsilon_{\wwG}(w)=0$ for every $w\in \wG$. 

\begin{theorem}[\text{\cite[Theorem 4.5]{M1}}]\label{thm:tensor_prod_decomp}
There exists an isomorphism of left $C$-super-comodule superalgebras
\begin{equation}\label{eq:psi}
\psi : \C \os{\simeq}{\longrightarrow} C \ot \wwG
\end{equation}
such that $(\varepsilon_C\ot \op{id}_{\wwG})\circ \psi: \C \to \wwG$, composed with the 
projection $\wwG \to \wedge^0(\wG) \oplus \wedge^1(\wG)=\Bbbk\oplus \wG$, coincides with the natural map 
$\C \to \C_0/\C_0^+\oplus \C_1/\C_0^+\C_1=\Bbbk\oplus \wG$. 
\end{theorem}

\begin{rem}\label{rem:tensor_prod_decomp}
(1)\ 
The properties of $\psi$ above implies that  
$(\varepsilon_C\ot \op{id}_{\wwG})\circ \psi: \C \to \wwG$, composed with the projection
$\wwG \to \wedge^0(\wG)=\Bbbk$, coincides with the canonical $\C \to \C_0/\C^+_0=\Bbbk$.
This is the same as saying that $\psi$ is counit-preserving, or explicitly, 
$\varepsilon_{\C}=(\varepsilon_{C}\ot \varepsilon_{\wwG})\circ \psi$. 
 
(2)\ The cited \cite[Theorem 4.5]{M1} states that there exists a counit-preserving isomorphism 
$\psi:\C \os{\simeq}{\longrightarrow} C \ot \wwG$ of left $C$-super-comodule superalgebras.  
The isomorphism constructed in the proof (see \cite[p.301, line 9]{M1}) is seen to have
the stronger property above.  One sees that $C \to \C_0$, $c \mapsto \psi^{-1}(c \otimes 1)$
gives $\xi$ such as in Lemma \ref{lem:CcoC}; it indeed gives rise to the isomorphism 
\eqref{eq:CcoC} by Sweedler's Hopf-Module Theorem \cite[Theorem 4.1.1]{Sw} (or \cite[1.9.4, p.15]{Mo}),
as is shown in {\cite[Footnote 5]{M3}}.  
Note that consequently, we have an isomorphism 
${}^{coC}\C \os{\simeq}{\longrightarrow} \wwG$
of superalgebras. 
\end{rem}


\subsection{The associated graded Hopf superalgebra}\label{subsec:graded_Hopf}
Retain $\G=\S \C$ as above.
Note that the construction of $\gr \A$ in \eqref{eq:grA} gives rise to an endo-functor
$\A \mapsto \gr \A$ on $\ms{SAlg}$ which preserves the tensor product. It then follows 
that 
\[ \gr \C=(\gr \C,\ \gr(\Delta_{\C}),\ \gr(\varepsilon_{\C}),\ \gr(\mathcal{S}_{\C})) \]
is a Hopf superalgebra. 
Notice from Theorem \ref{thm:tensor_prod_decomp} that $\C \simeq \gr \C$ as superalgebras;
but they are not necessarily isomorphic as Hopf superalgebras, see Lemma-Definition \ref{lemdef:graded} below.
Let 
\[ \gr \G=\S(\gr \C) \]
denote the affine algebraic supergroup represented by $\gr \C$. 
One sees that $\gr \C$ includes $C=(\gr\C)(0)$ as a Hopf sub-superalgebra, and 
the associated, quotient Hopf superalgebra $\gr \C/(C^+)$ is $\wwG$; see \cite[Proposition 4.9 (2)]{M1}. 
Let 
\begin{equation}\label{eq:q0q1}
q_0 : \gr \C \to (\gr\C)(0)=C,\quad q_1: \gr \C \to \gr \C/(C^+)=\wwG 
\end{equation}
denote the quotient maps. The right $C$-comodule structure $\wG \to \wG \ot C$ on $\wG$
which corresponds to the left co-adjoint action by $G$ uniquely extends to a right $C$-super-comodule 
superalgebra and super-coalgebra structure
\[ \wwG \to \wwG \ot C. \]
The associated super-coalgebra $C \rcosmash \wwG$ of smash coproduct \cite[p.207]{Mo}, being
the tensor product $C \ot \wwG$ as superalgebra, is a Hopf superalgebra which is canonically
isomorphic to $\gr \C$ through
\begin{equation}\label{eq:cano_isom_grC}
(q_0\ot q_1)\circ \gr(\Delta_{\C}): \gr \C \os{\simeq}{\longrightarrow} C \rcosmash \wwG;
\end{equation}
see \cite[Proposition 4.9 (2)]{M1}, again. 
In terms of supergroups,
the affine supergroup $\gr \G$ include $G$ and
$\S(\wwG)$ as closed sub-supergroups, so that $\S(\wwG)$ is normal, and the product morphism give
a  canonical isomorphism 
$G \ltimes \S(\wwG)\simeq \gr \G$. 

\begin{lemma}\label{lem:pre-canonical}
Choose an isomorphism $\psi: \C \os{\simeq}{\longrightarrow} C \ot \wwG$ such as in 
Theorem \ref{thm:tensor_prod_decomp}. 
Then the associated isomorphism $\gr \psi $ of graded superalgebras coincides with the canonical
isomorphism \eqref{eq:cano_isom_grC}.
\end{lemma}
\begin{proof}
This follows since the chosen $\psi$ is such that the graded superalgebra map $\gr \C \to \wwG$
associated with 
$(\varepsilon_C\ot \op{id}_{\wwG})\circ \psi : \C \to \wwG$ is the quotient map $q_1$ given in \eqref{eq:q0q1}. 
\end{proof}


\subsection{Harish-Chandra pairs}\label{subsec:Harsh-Chandra}
A \emph{Harish-Chandra pair} \cite[Section 6.1]{MS}
is a pair $(G,\ms{V})$ of an affine algebraic group $G$ and a right $G$-module $\ms{V}$,
which is equipped with a $G$-equivariant, symmetric bilinear map $[\ , \ ] : \ms{V}\times \ms{V}\to \Lie(G)$ 
such that $v \triangleleft [v,v]=0$ for all $v \in \ms{V}$. Here, $\Lie(G)$ is supposed to be a right $G$-module
by the adjoint action, and $\triangleleft$ indicates the $\Lie(G)$-action on $\ms{V}$ 
induced from the original $G$-action. 

To every affine algebraic supergroup $\G$, a Harish-Chandra pair $(G,\ms{V})$ is naturally assigned,
where $G$ is the associated affine algebraic group, $\ms{V}$ is the $\Lie(\G)_1=\wG^*$
given the right adjoint $G$-action, and $[\ , \ ]$ is the bracket of $\Lie(\G)$ restricted
to $\ms{V}=\Lie(\G)_1$. It is proved by \cite[Theorem 3.2]{M2} (see also \cite[Theorem 6.1]{MS}) 
that the assignment above
gives rise to an equivalence from the category of affine algebraic supergroups to the category of
Harish-Chandra pairs. This is a very useful result,
but it is used in this paper only at the following proof. 
We remark that the category-equivalence above is extended to those algebraic supergroups which are
not necessarily affine, as will be proved in the forthcoming \cite{MZ3}. 

\begin{lemmadef}\label{lemdef:graded}
For an affine algebraic supergroup $\G=\S \C$, the following are equivalent:
\begin{itemize}
\item[(i)] $\G \simeq \gr \G$ as affine supergroups;
\item[(ii)] $\C \simeq \gr \C$ as Hopf superalgebras;
\item[(iii)] The Hopf superalgebra map $q : \C \to C$ splits; 
\item[(iv)] The bracket on $\Lie(\G)$, restricted to $\Lie(\G)_1\times \Lie(\G)_1$, vanishes, or in notation,
$[\Lie(\G)_1,\ \Lie(\G)_1]=0$. 
\end{itemize}
\emph{If these equivalent conditions are satisfied, we say that $\G$ is} graded.
\end{lemmadef}
\begin{proof}
(i) $\Leftrightarrow$ (ii). \ This is obvious. 

(iv) $\Rightarrow$ (ii). 
This follows from the category-equivalence mentioned above, 
since the Harish-Chandra pair corresponding to $\gr \G$ is obtained from that pair of $\G$, just by
replacing the associated $[\ , \ ] :\Lie(\G)_1\times \Lie(\G)_1 \to \Lie(\G)_0$
with the zero map; see \cite[Section 4.6]{M2}, \cite[Section 4.2]{MS}. 

(ii) $\Rightarrow$ (iii).
Assume (ii).
Then the isomorphism $\C \os{\simeq}{\longrightarrow} \gr \C$, composed with the 
the natural Hopf superalgebra map $\gr(\C) \to \gr(\C)/(\gr(\C)_1)=C$ which obviously splits,
coincides with the composite of $q : \C \to C$  with some automorphism of $C$. This shows (iii). 

(iii) $\Rightarrow$ (iv).
Assume (iii). Since we then have the split exact sequence 
$0\to \Lie(\G)_0\to \Lie(\G)\to \Lie(\G)_1 \to 0$ of Lie superalgebras,
(iv) follows.  
\end{proof}


\section{The main theorem and its consequences}\label{sec:G/H}

This section is the main body of the paper.
Throughout, $\G=\S \C$ denotes an affine algebraic supergroup which includes 
a closed sub-supergroup $\H=\S \D$. 


\subsection{The key construction of open embeddings}\label{subsec:open_embed}
We have the closed embedding and the associated
surjection of Hopf superalgebras 
\begin{equation}\label{eq:GHCD}
\G \supset \H,\quad \C \to \D. 
\end{equation}
Since $\H$ acts freely on $\G$ by the right multiplication we can and we will discuss the quotient superscheme
$\G/\H$ and the faisceau $\G \tilde{/}\H$.  The results which
we are going to obtain for these have the obvious, opposite-sided analogues for $\H\, \backslash \G$ or $\H\, \tilde{\backslash} \G$,
which hold true, indeed. 

The second map in \eqref{eq:GHCD} induces a linear surjection
\begin{equation*}
\wG \to \wH.
\end{equation*} 
The kernel is denoted by
\begin{equation}\label{eq:Z}
\z=\z^{\G}_{\H}:= \mr{Ker}(\wG \to \wH).
\end{equation} 
Let $G=\S C$ and $H=\S D$ denote the affine algebraic groups associated with $\G$, $\H$, respectively.
We thus have 
\[C =\C/(\C_1),\quad D =\D/(\D_1).\]  
The embedding and the surjection in \eqref{eq:GHCD} induce
a closed embedding of affine algebraic groups and a Hopf-algebra surjection
\begin{equation}\label{eq:ghcd}
G \supset H,\quad C \to D.
\end{equation}
 
Here is a classical result; see \cite[Part I, Sections 5.6--5.7]{J}, for example.
There exists a (necessarily, unique) quotient scheme $G/H$, which is Noetherian, and
represents the faisceau $G\tilde{/}H$. 
The canonical morphism of schemes
\begin{equation*}\label{eq:pi}
\pi : G \to G/H
\end{equation*}
is affine, faithfully flat and finitely presented \cite[Part I, Section 5.7, (1)]{J}. 
Choose arbitrarily a non-empty affine open subscheme $U \subset G/H$. 
Then $\pi^{-1}(U)$ is an $H$-stable affine open subscheme of $G$ such that
\begin{equation*}
\pi^{-1}(U)/H=U,
\end{equation*}
and this represents the faisceau $\pi^{-1}(U)\tilde{/}H$. 
Suppose that $\pi^{-1}(U)=\S A$ and $U=\S B$. Then $A$ is fppf over $B$,
and $A\supset B$ is a $D$-Galois extension; see Definition \ref{def:Galois}.  

By the definition \eqref{eq:Z} we have the short exact sequence 
\begin{equation}\label{eq:ZWW}
0 \to \z\to \wG\to \wH\to 0
\end{equation}
of right $D$-comodules.

\begin{lemma}\label{lem:AZ}
The tensor product $A \ot \z$ of right $D$-comodules, given the obvious multiplication by $A$, 
turns into an object of $\SM_A^{D}$.  This naturally gives rise to the right $D$-super-comodule superalgebra
\[ A \otimes \zz\,  (=\wedge_A(A\ot \z)) \]
over $A$; notice from Remark \ref{rem:Schneider} (2) 
that such a superalgebra is precisely an algebra object of the tensor category $\SM_A^{D}$.
These $A \ot \z$ and $A \otimes \zz$ are injective as right $D$-comodules.
\end{lemma}
\begin{proof}
The assertions are obvious except the last. The last assertion follows since the right $D$-comodule
$A$ is injective. Indeed, 
an injective $D$-comodule tensored with any $D$-comodule is injective; see \cite[Part I, Section 3.10, 
Proposition c)]{J}. 
\end{proof}

Define superalgebras by
\begin{equation}\label{eq:AandB}
\A = (A \otimes \zz)\, \square_D\, \D,\quad \B= (A \otimes \zz)^{co D}.  
\end{equation}

\begin{prop}\label{prop:AandB}
We have the following.
\begin{itemize}
\item[(1)] $\A$ is naturally a right $\D$-super-comodule superalgebra such that
\begin{equation}\label{eq:AcoD}
\A^{co\D}= \B.
\end{equation}
Moreover, $\A$ is finitely generated as a superalgebra, and is injective as a $\D$-super-comodule. 
\item[(2)]
$\B$ is a graded subalgebra of $A \otimes \zz$, whose $0$-th component $\B(0)$ is $B$. The
first component 
\[ \B(1)=(A\ot \z)^{coD} \]
is a finitely generated projective $B$-module of constant rank $\dim \z$. The 
graded $B$-algebra map
\begin{equation}\label{eq:wedgeB(1)B}
\wedge_B(\B(1))\to \B 
\end{equation}
induced from the inclusion $\B(1)\subset \B$ is an isomorphism. Moreover, $\B$ is Noetherian. 
\end{itemize}
\end{prop}
\begin{proof}
(1)\ The first assertion easily follows once one sees
\[
\A^{co\D}= (A \otimes \zz)\, \square_D\, \D\, \square_{\D}\, \Bbbk=(A \otimes \zz)\, \square_D \, \Bbbk= \B.
\]

By Lemma \ref{lem:CcoC} applied to $\D$ we see that the inclusion 
$A \ot \zz \ot {}^{coD}\D \hookrightarrow A \ot \zz \ot \D$
gives the canonical isomorphism 
\begin{equation}\label{eq:AzcoDD}
A \ot \zz \ot {}^{coD}\D = \A.
\end{equation}
Since ${}^{coD}\D\, (\simeq \wwH)$ is finite-dimensional, $\A$ is finitely generated. 

Note that a $\D$-(super-)comodule is injective if and only if it is a direct summand of the direct sum of some copies of $\D$;
see Section \ref{subsec:super_vs_non-super}. 
One then sees that $\A$ is $\D$-injective since $A \otimes \zz$ is $D$-injective by the previous lemma.

(2)\ The first assertion is easy to see. Since $A\supset B$ is a $D$-Galois extension (Definition \ref{def:Galois}),
we have by Theorem \ref{thm:superOberst} (1) the category-equivalence $\SM_A^D \approx \SM_B$. 
This shows that the $A$-actions on $A \otimes \zz$, and on its
first component $A \otimes \z$ give isomorphisms 
\[ A \otimes_B \B \os{\simeq}{\tto} A \ot \zz,\quad A \ot_B \B(1) \os{\simeq}{\tto} A \ot \z \]
in $\SM_A^D$. Since $B \to A$ is faithfully flat, the second isomorphism shows
that $\B(1)$ is a $B$-module such as claimed above. 
The result implies that the first isomorphism above,
composed with the base extension of \eqref{eq:wedgeB(1)B} along $B \to A$, is an isomorphism.
Again by the faithful flatness, the map \eqref{eq:wedgeB(1)B} is an isomorphism.
It follows that $\B$ is Noetherian since $B$ is. 
\end{proof}

We will see in the proof of Corollary \ref{cor:quotient_local} that $\A \supset \B$ is 
a $\D$-Galois extension. 

\begin{rem}\label{rem:alternativeB}
Here are two alternative ways of describing $\B(1)$. 
\begin{itemize}
\item[(1)]
Given a right $D$-comodule, one has a left $D$-comodule, twisting the side of the coaction
through the antipode of $D$. Applied to $\z$, the resulting left $D$-coaction on $\z$ is what
corresponds to the right co-adjoint action by $H$; see Section \ref{subsec:Hopf}. 
Regrading $\z$ thus as a left $D$-comodule, we have
the alternative description
\begin{equation}\label{eq:alternativeB1}
\B(1)=A\, \square_D\, \z, 
\end{equation}
which will be often used. By Proposition \ref{eq:AandB} (2) we have
\begin{equation}\label{eq:B}
\B \simeq \wedge_B(A \square_D\, \z).
\end{equation}
\item[(2)]
By \eqref{eq:Lie1_W} the dual $\z^*$ of $\z$ is the quotient vector space 
$\op{Lie}(\G)_1/\op{Lie}(\H)_1$, on which $H$ acts by adjoint from the left. We see that
$\B(1)$ is identified so as
\begin{equation}\label{eq:alteranativeB2}
\B(1)=\ms{Comod}^D(\z^*, A)\, (=\ms{Mod}_H(\z^*, A))
\end{equation}
with the vector space of right $D$-comodule (or left $H$-module) maps.
\end{itemize}
\end{rem}

We regard the tensor product $A \ot \wG$ of right $D$-comodules 
as a purely odd object of $\SM_A^D$, with respect to the obvious multiplication by $A$; 
this then includes $A \ot \z$ as a sub-object. 

\begin{lemma}\label{lem:split}
The inclusion $A \ot \z \hookrightarrow A \ot \wG$ splits in $\SM_A^D$. 
\end{lemma}
\begin{proof}
Since $A$ is $D$-injective, the unit map $\Bbbk \to A$, which is $D$-colinear, extends to a $D$-comodule
map
\begin{equation}\label{eq:eta}
\eta : D \to A,
\end{equation}
which thus satisfies $\eta(1)=1$. It follows by \cite[Theorem 1]{D} that 
a short exact sequence in $\SM_A^D$ splits if it splits $A$-linearly. 
Since obviously, the inclusion in question splits $A$-linearly, it splits in $\SM_A^D$, as desired.
See the following remark for the explicit retraction constructed from $\eta$. 
\end{proof}

\begin{rem}\label{rem:compatible}
(1)\ 
Let $\eta : D \to A$ be as in \eqref{eq:eta}.
Given an object $M$ of $\SM_A^D$, let 
$M \to M\ot D$, $m\mapsto m_{(0)}\ot m_{(1)}$ denote the $D$-comodule structure map.
This is a monomorphism in $\SM_A^D$, and its retraction is given by
\[ \sigma_M : M \ot D \to M, \ \sigma_M(m \ot d) = m_{(0)}\eta(\mathcal{S}_D(m_{(1)})d), \]
where $\mathcal{S}_D$ denotes the antipode as in \eqref{eq:Hopf_struc}; see \cite[Page 100, line --1]{D}.  
One sees easily that a retraction of 
the inclusion $A \ot \z \hookrightarrow A \ot \wG$ above is given by
the composite
\begin{equation}\label{eq:retraction}
A \ot \wG \to (A \ot \wG) \ot D \os{\mr{id} \ot r\ot \mr{id}}{\tto} 
(A \ot \z)\ot D \os{\sigma_{\hspace{-1mm}A\ot \z}}{\tto} A \ot \z, 
\end{equation}
where the first arrow is the $D$-comodule structure map on $A \ot \wG$, 
and $r : \wG \to  \z$ is an arbitrarily
chosen, linear retraction of the inclusion $\z \hookrightarrow \wG$. 

(2)\
Let $(\O_G(\pi^{-1}(U))=)\, A \to A'=\O_G(\pi^{-1}(U'))$ be the restriction map
associated with $\pi^{-1}(U) \supset \pi^{-1}(U')$, 
where $U'$ is any non-empty affine open subscheme of $G/H$ included in $U$.
It then follows that retractions such as above can be chosen so as to be compatible 
with $A \to A'$, since $\eta$ can be so chosen. 
\end{rem}

Let us choose a retraction in $\SM_A^D$
\begin{equation}\label{eq:theta}
\theta : A \ot \wG \to  A \ot \z
\end{equation}
of the inclusion $A \ot \z \hookrightarrow A \ot \wG$; it may not be  
such as above that was constructed from some $\eta$. 

Recall that $A$ is an (algebra) object of $\SM_A^D$.
We regard $\wG \ot A$ as such an object with respect to the
structure possessed by the tensor factor $A$. 
Recall that $\pi^{-1}(U)=\S A$ is an affine open subset of $G=\S C$. Let
$\iota : C \to A$ is the algebra map associated with $G \supset \pi^{-1}(U)$. 

\begin{lemma}\label{lem:kappa}
The map
\begin{equation}\label{eq:kappa}
\kappa=\kappa_A : A \ot \wG\to \wG \ot A,\quad \kappa(a \ot w) =w_{(0)}\ot a\, \iota (w_{(1)})
\end{equation}
is an isomorphism in $\SM_A^D$, where
$w \mapsto w_{(0)}\ot w_{(1)}$ indicates the right $C$-comodule structure map
$\wG \to \wG \ot C$. 
\end{lemma}
\begin{proof}
Indeed,
$w \ot a \mapsto a\, \iota(\mathcal{S}_D(w_{(1)}))\ot w_{(0)}$ gives an inverse. 
\end{proof}

\begin{rem}\label{rem:C-costable}
Assume that $\z$ is a $C$-subcomodule (or equivalently, a $G$-submodule) of $\wG$; 
by \cite[Lemma 3.5]{MZ2} this is satisfied, if $\H$ is normal in $\G$, or namely, if 
for every $\mathbb{T}\in \ms{SAlg}$, $\H(\mathbb{T})$ 
is normal in $\G(\mathbb{T})$.
Since the restriction $\kappa|_{A \ot \z}$ of $\kappa$ to $A \ot \z$ 
then maps isomorphically onto $\z \ot A$, we have
\begin{equation*}\label{eq:B(1)}
\B(1) \simeq \z \ot A^{coD}=\z \ot B,
\end{equation*}
so that $\B(1)$ is a free $B$-module of rank $\dim \z$; cf. \eqref{eq:B(1)}. 
\end{rem}

Let 
\begin{equation}\label{eq:theta'}
\theta' : \wG \ot A \os{\kappa^{-1}}{\tto} A\ot \wG \os{\theta}{\tto}  A \ot \z
\end{equation}
be the composite of $\kappa^{-1}$ with the $\theta$ chosen before. 
This is thus a retraction of $\kappa|_{A\ot \z} : A \ot \z \to \wG \ot A$ in $\SM_A^D$. 
There arises the $A$-algebra morphism 
\[ \wedge(\theta'): \wwG \ot A \to A \ot \zz \]
in the tensor category $\SM_A^D$ (see Remark \ref{rem:Schneider} (2)), which is a retraction of $\wedge(\kappa|_{A \ot \z})$. 
Essentially by Theorem \ref{thm:tensor_prod_decomp} we can choose an isomorphism
\[ \psi': \C \os{\simeq}{\tto} \wwG \ot C \]
with the analogous, opposite-sided properties to those ones
which $\psi : \C \os{\simeq}{\tto} C\ot \wwG$ such as in \eqref{eq:psi} has. 
We define
\begin{equation}\label{eq:omega}
\omega_{\theta}: \C \to \A=(A\ot \zz)\, \square_D\, \D
\end{equation}
to be the composite
\begin{align}\label{eq:first_row}
\C &\os{\Delta_\C}{\tto} \C\, \square_D\, \C \tto \C\, \square_D\, 
\D \overset{\psi'\square\mr{id}}{\longrightarrow} (\wwG\ot C)\, \square_D\, \D\\ 
&\os{(\op{id}\otimes \iota) \square \op{id}}{\tto} (\wwG \ot A)\, \square_D\, \D 
\overset{\wedge(\theta')\square \mr{id}}{\longrightarrow} (A \ot \zz)\, \square_D\, \D, \notag
\end{align}
where the second arrow is the Hopf algebra quotient $C\to D$
co-tensored with $\mr{id}_{\C}$.
As for the first arrow note that the coproduct $\Delta_\C$ goes into the co-tensor product $\C\, \square_D\, \C$. 
As for the third, $\psi'$, being $C$-colinear, is $D$-colinear.

\begin{prop}\label{prop:open_embed}
$\omega_{\theta} : \C \to \A$ gives rise to a right $\H$-equivariant embedding
\[ \S(\omega_{\theta}):\S \A \to \S \C=\G. \]
of superschemes onto the open subset $\pi^{-1}(U)$ of $|\G|\, (=|G|)$. 
\end{prop}
\begin{proof}
For simplicity let us write $\X$ for $\S \A$.
As is seen from \eqref{eq:AzcoDD}, 
the underlying topological space $|\X|=\S(\A_0)$ of $\X$ is naturally identified with $\pi^{-1}(U)=\S A$. 
Let $\O_{\G}$ denote the structure sheaf of $\G$.
We should prove the following two:
\begin{itemize}
\item[(1)]\
The algebra map associated with $\omega_{\theta}$ coincides with $\iota : C\to A$, so that
$\S(\omega_{\theta})$ gives the open embedding $\pi^{-1}(U) \hookrightarrow |\G|\, (=|G|)$ 
of the underling topological spaces;
\item[(2)]\
The restricted sheaf $\O_{\G}|_{|\X|}$
coincides with $\O_{\X}$.
\end{itemize}

We wish to see what the graded superalgebra map $\gr(\omega_{\theta})$ associated with $\omega_{\theta}$ is. 
By the analogous, opposite-sided result to Lemma \ref{lem:pre-canonical}, $\gr(\psi')$ is the canonical 
isomorphism $\gr \C = \wwG \lcosmash C \os{=}{\tto} \wwG \ot C$.
Therefore, the graded algebra map associated with the first row \eqref{eq:first_row}
of the composite defining $\omega_{\theta}$ is 
\[
\gr \C=C\ot \wwG \to C \ot \wwG \ot \wwH \os{\wedge(\kappa_C)\ot \op{id}}{\tto}
\wwG \ot C \ot \wwH, 
\]
where the first arrow is the natural right $\wwH$-super-comodule structure map, and the 
$\wedge(\kappa_C)$ in the second arrow is the graded-algebra isomorphism arising from the 
isomorphism $\kappa_C$ as defined by \eqref{eq:kappa}.  
By using $\kappa_A\circ(\iota \ot \op{id}_{\wG}) =(\op{id}_{\wG}\ot\, \iota)\circ \kappa_C$, 
it follows that $\gr(\omega_{\theta})$ is the composite
\begin{equation*}\label{eq:gr_isom}
\begin{aligned}
\gr \C&=C\ot \wwG \to C \ot \wwG \ot \wwH\\ 
&\os{\iota \ot \op{id}\ot \op{id}}{\tto}
A \ot \wwG \ot \wwH
\os{\wedge(\theta)\ot \op{id}}{\tto} 
A \ot \zz \ot \wwH.
\end{aligned}
\end{equation*}
This is seen to be $\iota : C \to A$ in degree zero. This proves (1).

To prove (2), 
let $P \in \S(\A_0)$, and set $Q=\omega_{\theta}^{-1}(P)$. 
We should prove that the local superalgebra map of stalks
\begin{equation}\label{eq:omega_theta_P}
(\omega_{\theta})_P : \C_Q \to \A_P, 
\end{equation}
at $P$ is an isomorphism.  
It suffices to prove that the associated graded algebra map $\gr((\omega_{\theta})_P)$ 
is an isomorphism.  
As was seen in the proof of Lemma \ref{lemdef:graded_superscheme}, we have
$\gr((\omega_{\theta})_P)=\gr(\omega_{\theta})_P$.
As for the latter, we may suppose $P\in \S A$, $Q\in \S C$ and that 
the relevant localizations $(\gr \A)_P$ and $(\gr \C)_Q$ are by those.
In the same situation, $\iota_P : C_Q \to A_P$ is an isomorphism since $\pi^{-1}(U) \subset G$ is open. 
From the result obtained in the preceding paragraph we see that
\[ \gr(\omega_{\theta})_P: C_Q \ot \wwG \to A_P \ot \zz \ot \wwH. \]
is a right $\wwH$-super-comodule superalgebra map, which, restricted to
the $\wwH$-coinvariants, coincides with $\iota_P \ot \op{id} : C_Q \ot \zz \os{\simeq}{\tto} A_P \ot \zz$.
This last property shows that $\gr(\omega_{\theta})_P$ is an isomorphism, as desired. 
Indeed, 
$\gr(\omega_{\theta})_P$ is a morphism in that category $\SM_{\wwG}^{\wwH}$ which arises from the right $\wwH$-super-comodule
superalgebra $\wwG$; the map is, moreover, a $\wwG$-algebra morphism in the category. 
Since $\wwG \supset \zz$ is obviously a $\wwH$-Galois extension, satisfying Condition (iii) of Theorem \ref{thm:superOberst} (1),
the resulting category-equivalence $\SM_{\wwG}^{\wwH}\approx \SM_{\zz}$ can apply to see the result.
\end{proof}

\begin{rem}\label{rem:open_embed}
The argument of the last proof
(see also the proof of Lemma \ref{lemdef:graded_superscheme}) shows the following. Given a superalgebra map
$\omega : \mathbb{T}\to \mathbb{S}$ between Noetherian superalgebras, the associated morphism
$\S(\omega) :\S \mathbb{S} \to \S \mathbb{T}$ of affine superschemes is an open embedding if and only if the morphism 
$\S(\gr(\omega)):\S(\gr \mathbb{S}) \to \S(\gr \mathbb{T})$ associated with $\gr(\omega) : \gr \mathbb{T}\to \gr \mathbb{S}$
is an open embedding. Note that the two associated continuous maps between the underlying topological spaces are naturally
identified. This result is easily generalized in the obvious manner to morphisms of Noetherian superschemes.
\end{rem}


\subsection{The main theorem}\label{subsec:main_thm} 
Retaining the situation as above we have the following corollary to the previous proposition.

\begin{corollary}\label{cor:quotient_local}
The open subset $\pi^{-1}(U)\subset |\G|$, regarded as an open sub-superscheme of $\G$,
is an $\H$-equivariant affine
superscheme $\S(\O_{\G}(\pi^{-1}(U)))$ such that the faisceau 
$\pi^{-1}(U)\tilde{/}\H$ is the affine superscheme
\[ \S \big(\O_{\G}(\pi^{-1}(U))^{co \D}\big), \] 
which is Noetherian.
\end{corollary}
\begin{proof}
By Proposition \ref{prop:open_embed}, $\pi^{-1}(U)$ is $\H$-stable and affine in $\G$; in particular, 
$\O_{\G}(\pi^{-1}(U))$ is a right $\D$-super-comodule superalgebra. Moreover, $\omega_{\theta}$ naturally 
factors through an isomorphism
$\O_{\G}(\pi^{-1}(U))\os{\simeq}{\tto} \A$
of $\D$-super-comodule superalgebras, which obviously restricts to
$\O_{\G}(\pi^{-1}(U))^{co\D}$ $\os{\simeq}{\tto} \B$.
Recall from Proposition \ref{prop:AandB} that $\A$ is finitely generated, and $\B$ is Noetherian.

We claim that $\A \supset \B$ is a $\D$-Galois extension; this implies the corollary by Theorem \ref{thm:superOberst} (2).
Since $\A$ is $\D$-injective by Proposition \ref{prop:AandB} (1), it suffices by Theorem \ref{thm:superOberst} (1) (see Condition (ii))
to prove that 
the $\A_0$-superalgebra map $\alpha_{\A}$ in \eqref{eq:alpha} is surjective. Let $P$, $Q$ be as in the last proof. 
Then we have the following commutative diagram which contains the map $(\omega_{\theta})_P$ in \eqref{eq:omega_theta_P};
it has been proved to be an isomorphism. 
\[
\begin{xy}
(0,16)   *++{\C_Q \ot \C}  ="1",
(30,16)  *++{\C_Q \ot \D}    ="2",
(0,0)   *++{\A_P \ot \A}  ="3",
(30,0)   *++{\A_P \ot \D}  ="4",
{"1" \SelectTips{cm}{} \ar @{->}^{(\alpha_{\C})_Q} "2"},
{"1" \SelectTips{cm}{} \ar @{->}_{(\omega_{\theta})_P\ot \omega_{\theta}}"3"},
{"2" \SelectTips{cm}{} \ar @{->}^{(\omega_{\theta})_P\ot \op{id}}_{\simeq} "4"},
{"3" \SelectTips{cm}{} \ar @{->}^{(\alpha_{\A})_P} "4"}
\end{xy}
\]
Here the horizontal arrows are localizations of the alpha maps.
The upper $(\alpha_{\C})_Q$ is surjective since the map $\alpha_\C$ factors through the canonical isomorphism
$\C \ot \C \os{\simeq}{\longrightarrow} \C\ot \C,\ x \ot y \mapsto x\Delta_{\C}(y)$,  and is, therefore,
surjective. 
It follows that the lower $(\alpha_{\A})_P$ is as well, proving the desired surjectivity.
\end{proof}
\begin{rem}\label{rem:underlying_space}
As for the isomorphism $\mathcal{O}_{\G}(\pi^{-1}(U)) \overset{\simeq}{\longrightarrow} \A$, the induced isomorphism
of the associated algebras is the identity map of $A$; see the first half of the proof of 
By Proposition \ref{prop:open_embed}. Therefore, as for the restricted isomorphism 
$\O_{\G}(\pi^{-1}(U))^{co\D} \overset{\simeq}{\longrightarrow} \B$, that isomorphism is the identity map of $B$. 
Hence
the underlying topological space of $\S(\O_{\G}(\pi^{-1}(U))^{co\D})$ coincides with $U$, that space of $\S B$. 
\end{rem} 

Given a non-empty affine open subset $U$ of $|G/H|$, we thus have
the Noetherian affine superscheme $\S \big(\O_{\G}(\pi^{-1}(U))^{co \D}\big)$ with underlying topological space $U$.  

\begin{theorem}\label{mainthm}
The Noetherian affine superschemes 
\[ \S\big(\O_{\G}(\pi^{-1}(U))^{co \D}\big), \] 
where $U$ ranges over non-empty
affine open subsets of $|G/H|$, are uniquely glued into a superscheme, which is Noetherian, 
with the underlying topological space $|G/H|$. 
This superscheme is the quotient superscheme $\G/\H$
of $\G$ by $\H$, and represents 
the faisceau $\G\tilde{/}\H$. 
\end{theorem}
\begin{proof}
The theorem consists of two assertions. 

\emph{Proof of the first assertion.}\
Let $U =\S B \supset U'=\S B'$ be affine open subschemes of $G/H$. 
The restriction map $\O_{\G}(\pi^{-1}(U))\to \O_{\G}(\pi^{-1}(U'))$ restricts to
\begin{equation}\label{eq:restriction_map}
\O_{\G}(\pi^{-1}(U))^{co\D}\to \O_{\G}(\pi^{-1}(U'))^{co\D}.
\end{equation}
Suppose $\pi^{-1}(U)=\S A$,\ $\pi^{-1}(U')=\S A'$ in $G$. From these $A$ and $A'$,
we construct superalgebras
$\A \supset \B$, $\A'\supset \B'$, respectively, as in \eqref{eq:AandB}. By choosing 
retractions $A^{(')} \ot \wG \to A^{(')} \ot \z$ as in \eqref{eq:retraction}, we construct
$\D$-super-comodule superalgebra maps
\[ \C \to \A =(A\ot \zz)\, \square_D \, \D,\quad \C \to \A' =(A'\ot \zz)\, \square_D \, \D, \]
as in \eqref{eq:omega}, which give open embeddings of $\S \A$ and of $\S \A'$ into $\G=\S \C$ by Proposition \ref{prop:open_embed}. 
As is seen from Remark \ref{rem:compatible} (2) and the description \eqref{eq:retraction}
of the retractions, we may suppose that the two superalgebra maps above are
compatible with the map $(A\ot \zz)\, \square_D\, \D \to (A'\ot \zz)\, \square_D\, \D$ which arises
from the restriction map $A \to A'$; this compatibility is expressed by commutativity of the diagram: 
\[
\begin{xy}
(0,0)   *++{\C}  ="1",
(16,8)  *++{}  ="1.5",
(33.5,8)  *++{\A =(A\ot \zz)\, \square_D\, \D}    ="2",
(16,-8)  *++{}  ="2.5",
(33.5,-8)   *++{\A' =(A'\ot \zz)\, \square_D \, \D}  ="3",
{"1" \SelectTips{cm}{} \ar @{->} "1.5"},
{"1" \SelectTips{cm}{} \ar @{->} "2.5"},
{"2" \SelectTips{cm}{} \ar @{->} "3"}
\end{xy}
\]
Consequently, the map \eqref{eq:restriction_map}
may be supposed to be the map
\begin{equation}\label{eq:BtoB'}
\B = A\, \square_D\, \zz\, \to\, 
A'\, \square_D\, \zz =\B'
\end{equation}
which arises from $A \to A'$, again. Here for $\B$ and $\B'$,  
we have used description analogous to \eqref{eq:alternativeB1}. 

Let $P \in U'\, (=\S B')$. Let $Q$ be the pullback of $P$ in $B$ along 
the algebra map $B \to B'$ associated with $U \supset U'$. The map induces
an isomorphism, $B_Q \os{\simeq}{\tto} B'_P$, of stalks. 
We claim that the superalgebra map above induces an isomorphism, $\B_Q \os{\simeq}{\tto} \B'_P$,
between the stalks. Here one should notice from \eqref{eq:B} that $\B_Q=\B\ot_B B_Q$ and $\B'_P=\B'\ot_{B'}B'_P$ 
are indeed the stalks; see the proof of Lemma \ref{lemdef:graded_superscheme}. 
Since $B^{(')} =A^{(')coD}$, the exactness of localization shows that the localized $\B_Q \to \B'_P$ coincides
with the localized $A_Q \to A'_P$ co-tensored over $D$ with the identity map
$\zz \to \zz$. 
The map $A_Q \to A'_P$ is a $D$-comodule algebra map, and it restricts to the isomorphism
$B_Q \os{\simeq}{\tto} B'_P$. Since $A_Q\supset B_Q$ and $A'_P\supset B'_P$ are $D$-Galois, it follows that 
$A_Q \to A'_P$ is an isomorphism, proving the claim.
Indeed, we have the commutative diagram:
\[
\begin{xy}
(0,16)   *++{A'_P\ot_{B_Q}A_Q}  ="1",
(30,16)  *++{A'_P \ot D}    ="2",
(0,0)   *++{A'_P\ot_{B'_P}A'_P}  ="3",
(30,0)   *++{A'_P \ot D}  ="4",
{"1" \SelectTips{cm}{} \ar @{->}^{\quad \simeq} "2"},
{"1" \SelectTips{cm}{} \ar @{->} "3"},
{"2" \SelectTips{cm}{} \ar @{->}^{\mr{id}} "4"},
{"3" \SelectTips{cm}{} \ar @{->}^{\quad \simeq} "4"}
\end{xy}
\]
Here the second row is the canonical isomorphism $\beta$ for $A'_P\supset B'_P$ (see \eqref{eq:beta}), while
the first is the base extension of the isomorphism for $A_Q\supset B_Q$, along $A_Q \to A'_P$. 
Since $(B_Q\simeq)\, B'_P\to A'_P$ is faithfully flat, $A_Q \to A'_P$ is an isomorphism.

Let
\[ \Y_U:=\S\big(\O_{\G}(\pi^{-1}(U))^{co \D}\big). \]
The claim just proven implies that the structure sheaf of this $\Y_U$,
restricted to $U'$, coincides with that sheaf of $\Y_{U'}$.
This proves the first assertion: the Noetherian affine superschemes are uniquely glued into a superscheme,
say $\Y$. 
It is Noetherian since $G/H$, being Noetherian, is covered by finitely many $U$'s.  

\emph{Proof of the second assertion.}\
By Corollary \ref{cor:quotient_local}, $\Y_U$ represents the faisceau $\pi^{-1}(U)\tilde{/}\H$.
By Lemma \ref{lem:X/G} $\Y_U$ is the quotient superscheme $\pi^{-1}(U)/ \H$. 
We know that $\Y$ is the union $\bigcup_{i}\Y_{U_i}\, (= \bigcup_{i}\pi^{-1}(U_i)/ \H)$, 
where $|G/H|=\bigcup_iU_i$. As is easily seen, the quotient morphisms $ \pi^{-1}(U_i) \to \Y_{U_i}$ uniquely 
extend to a morphism $\G \to \Y$. 
It follows that the superscheme $\Y$ equipped with the last morphism 
is the quotient superscheme  $\G/\H$. 
The category-equivalence $\X \mapsto \X^{\diamond}$ in Theorem \ref{thm:comparison} preserves
open embeddings, and $\Y^{\diamond}= \bigcup_{i}\pi^{-1}(U_i)\tilde{/}\H$; see \cite[Lemma 5.2]{MZ1}.  
It follows that $\Y^{\diamond}=\G\tilde{/}\H$, or $\Y$ represents $\G\tilde{/}\H$. 
As an additional remark, $\Y_U$ represents, indeed, the faisceau dur $\pi^{-1}(U)\tilde{\tilde{/}}\H$, which
coincides with the faisceau $\pi^{-1}(U)\tilde{/}\H$, as is seen from the proof of Corollary \ref{cor:quotient_local}.
Therefore, the argument above shows $\Y^{\diamond}=\G\tilde{\tilde{/}}\H=\G\tilde{/}\H$.
\end{proof}

\begin{rem}\label{rem:about_mainthm}
Let $\O_{\G/\H}$ (resp., $\O_{G/H}$) denote the structure sheaf of $\G/\H$ (resp., $G/H$). 
In view of \eqref{eq:B} we see from the last proof that $\O_{\G/\H}$ is \emph{locally} isomorphic to
\[ \wedge_{\O_{G/H}}(\pi_*\O_G\, \square_{D}\, \z); \]
to be more precise the two sheaves are isomorphic, restricted to every open subset that is affine
in $G/H$, or equivalently, in $\G/\H$; see Proposition \ref{prop:affinity2} (2) below. 
Here note that for every open subset $U \subset |G/H|\, (=|\G/\H|)$, 
$\pi^{-1}(U)\, (\simeq U \times_{G/H} G)$ is 
$H$-stable in $G$.
Hence $\O_G(\pi^{-1}(U))$ is naturally a right $D$-comodule, and so we have the co-tensor product 
$\O_G(\pi^{-1}(U))\, \square_{D}\, \z$, which is naturally identified with 
the super-vector space
\[ \operatorname{Hom}_{D^*}(\z^*, \O_G(\pi^{-1}(U))) \]
of the maps $\z^*\to \O_G(\pi^{-1}(U))$ of the left supermodules
over the dual superalgebra $D^*$ of $D$; see \eqref{eq:alteranativeB2}. 
This shows that the presheaf
$\pi_*\O_G\, \square_{D}\, \z$, which assigns $\O_G(\pi^{-1}(U))\, \square_{D}\, \z$
to every open $U$, is a sheaf. It is indeed an $\O_{G/H}$-module sheaf, which is locally free,
as is seen from Proposition \ref{prop:AandB} (2).
\end{rem}


\subsection{Some consequences of the theorem}\label{subsec:consequences}
The first half of the next corollary has been obtained in \cite[Corollary 8.15]{MZ1}, while the second half is probably new.
Recall from \eqref{eq:GHCD} and \eqref{eq:ghcd} the notation. 

\begin{corollary}\label{cor:affinity}
The superscheme $\G/\H$ is affine if and only if the scheme $G/H$ is affine.
In this case we have
\begin{itemize}
\item[(1)] $\C\simeq (C \otimes \zz) \square_D \D$ as right $\D$-super-comodule superalgebras, and
\item[(2)] $\G/\H=\S(\wedge_B(C\, \square_D\, \z))$, 
\end{itemize}
where we let $B=C^{co D}$, and so  $G/H=\S B$. 
\end{corollary}
\begin{proof}
The ``only if" follows since one sees from the functorial viewpoint that
if $\G\tilde{/}\H=\op{Sp} \B$ is affine, then $G\tilde{/}H=\op{Sp}(\B/(\B_1))$.
The remaining follows
from Proposition \ref{prop:open_embed}, Corollary \ref{cor:quotient_local} and Remark \ref{rem:about_mainthm}.  
\end{proof}

\begin{rem}\label{rem:affinity}
It follows by \cite[Theorem 5.2]{Z} or Theorem \ref{thm:superOberst} that $\G/\H$ is 
affine if and only if
\begin{itemize}
\item[(i)] $\C$ is injective (or equivalently, coflat) as a left or right $\D$-comodule.
\end{itemize}
It is known (see \cite[Theorem 5.9]{M1}, \cite[Theorem 6.2]{Z}) that if $\H$ is normal in $\G$,  then
the equivalent conditions are satisfied, and $\G/\H$ is naturally an affine algebraic supergroup.  
The classical counterpart (see \cite[Theorem 10]{T2}, \cite[Folgerung B]{O}) states that $G/H$ is affine if and only if
\begin{itemize}
\item[(ii)] $C$ is injective (or equivalently, coflat) as a left or right $D$-comodule.
\end{itemize}
If $H$ is normal in $G$,  then
the equivalent conditions are satisfied, and $G/H$ is naturally an affine algebraic group.  

Therefore, Corollary \ref{cor:affinity}
tells us that Conditions (i) and (ii) are equivalent. If $H$ is normal in $G$, then $\G/\H$ is an affine
algebraic superscheme since the $B$ is in the corollary is then finitely generated. 
\end{rem}

The first half of the last corollary (or \cite[Corollary 8.15]{MZ1}) is generalized by Part 2 of the next Proposition, which would
reveal a remarkable feature of $\G/\H$. 

\begin{prop}\label{prop:affinity2}
We have the following.
\begin{itemize}
\item[(1j]
$G/H$ is naturally isomorphic to the scheme associated with the superscheme $\G/\H$; see Definition \ref{lemdef:associated_scheme}. 
\item[(2)]
Given an open subset $U\subset |\G/\H|\, (=|G/H|)$,  
the open sub-superscheme $(U,\O_{\G/\H}|_U)$ of $\G/\H$ is affine if and only if the open subscheme 
$(U, \O_{G/H}|_U)$ of $G/H$ is affine.
\item[(3)]
$\G/\H$ is smooth if and only if $G/H$ is smooth. 
The equivalent conditions hold, either if the characteristic $\operatorname{char} \Bbbk$ of $\Bbbk$ is zero,
or if $\operatorname{char} \Bbbk>2$ and $G$ is smooth. 
\item[(4)]
The quotient morphism $\G \to \G/\H$ is affine, faithfully flat and finitely presented.
\end{itemize}
\end{prop}
\begin{proof}
(1)\
This was shown in \cite[Proposition 9.3]{MZ1}, but now it turns much easier to prove. Indeed,
the structure sheaf of the associated scheme is naturally isomorphic to $\O_{G/H}$, since it is so
on all affine open subsets of $|G/H|$ by Remark \ref{rem:underlying_space}.

(2)\ 
The ``if" follows from Theorem \ref{mainthm}, while the ``only if" follows from Part 1 above.

(3)\ For every point $P \in  |\G/\H|$, $\O_{\G/\H ,P}$ is the exterior algebra over $\O_{G/H,P}$ 
on a finitely generated free $\O_{G/H,P}$-module, as is seen from Proposition \ref{prop:AandB} (2). 
It follows from \cite[Theorem A.2]{MZ2} that $\O_{\G/\H ,P}$ is smooth if and only if $\O_{G/H,P}$ is. 
This proves the first assertion. 
The second follows from the facts: (i)~$G/H$ is smooth if $G$ is smooth, and (ii)~the last assumption is
always satisfied if $\operatorname{char} \Bbbk=0$. 

(4)\
In addition to Part 2 above we have the result in a special case that an open sub-superscheme of $\G$ is affine 
if and only if the associated, open sub-scheme of $G$ is affine.
Hence the desired affinity follows from the fact that $G \to G/H$ is affine. 
The faithful flatness follows since 
with the notation \eqref{eq:AandB}, $\A$ is faithfully flat over $\B$.
The remaining follows, since the $\B$-superalgebra $\A$ is finitely presented, 
being so after base extension along $\B \to \A$; this last is seen from the isomorphism 
$\beta: \A\ot_{\B}\A\os{\simeq}{\tto}\A\ot \D$ in \eqref{eq:beta}.  
\end{proof}

\begin{rem}\label{rem:from_MZ}
Corollary 9.10 of \cite{MZ1} proves that 
the first two properties of Part 4 above are possessed, more generally, by
the quotient superscheme $\X/\G$, if it exists and represents the faisceau $\X\tilde{/}\G$, where $\X$ is 
an affine superscheme, and $\G$ is
an affine algebraic supergroup which freely acts on $\X$.
\end{rem}

Recall from Section \ref{subsec:graded_superscheme} the definition of superschemes being split.

\begin{prop}\label{prop:grG/H}
The graded superscheme $\gr(\G/\H)$ associated with $\G/\H$ (see Definition \ref{lemdef:graded_superscheme})
is split with the structure sheaf 
\begin{equation}\label{eq:splitting_module_sheaf}
\wedge_{\O_{G/H}}(\pi_*\O_G\, \square_{D}\, \z);
\end{equation}
see Remark \ref{rem:about_mainthm}. Moreover, the morphism $\gr \G \to \gr(\G/\H)$ associated
with the quotient morphism $\G \to \G/\H$ induces an isomorphism
\begin{equation}\label{eq:isom_gr}
\gr \G/\gr \H \simeq \gr(\G/\H). 
\end{equation}
\end{prop}

\begin{proof}
Recall that the structure sheaf $\O_{\gr(\G/\H)}$ of $\gr(\G/\H)$ is a sheaf of graded superalgebras. 
The $0$-th component is $\O_{G/H}$ by Proposition \ref{prop:affinity2} (1).  
The first component coincides with $\pi_*\O_{G/H}\square_{D}\, \z$, since it does on
all affine open subsets of $|G/H|$; see Remark \ref{rem:about_mainthm}. Therefore, we have a natural
morphism $\wedge_{\O_{G/H}}(\pi_*\O_G\, \square_{D}\, \z)\to \O_{\gr(\G/\H)}$ of sheaves, which is identical in degree $0$, $1$; 
this is isomorphic since it is so on all affine open sets. 
The result just proven, combined with Proposition \ref{prop:split} bellow in Case (c) (applied to $\gr \G$), proves \eqref{eq:isom_gr}. 
\end{proof}

\begin{prop}\label{prop:split}
The superscheme $\G/\H$ is split with the structure sheaf as in \eqref{eq:splitting_module_sheaf},
either if
\begin{itemize}
\item[(a)] $\z$ is a $C$-subcomodule of $\wG$, 
\item[(b)] $\z$ is a $D$-comodule direct summand of $\wG$, or
\item[(c)] $\G$ is graded in the sense as defined by Definition \ref{lemdef:graded}.  
\end{itemize}
\end{prop}

To prove this we wish to show that
the sheaf \eqref{eq:splitting_module_sheaf} and  $\O_{\G/\H}$ are naturally isomorphic on all affine
open subsets of $|G/H|$, 
in Cases (a), (b) and in Case (c), separately. 

\begin{proof}[Proof in Cases (a), (b)]
In these cases we construct retractions $\theta : A \ot \wG \to A \ot \z$ as in \eqref{eq:theta},
which do not depend on the $\eta$ in \eqref{eq:eta}.  
In Case (a), choose a linear retraction $r : \wG \to \z$ of the inclusion $\z \hookrightarrow \wG$,
and define $\theta$ to be the composite
\[
A \ot \wG \os{\kappa^{-1}}{\tto} \wG \ot A \os{r\ot \op{id}}{\tto} 
\z \ot A \os{(\kappa|_{A \ot \z})^{-1}}{\tto} A \ot \z.
\]
This is indeed possible as is seen from Remark \ref{rem:C-costable}. 
In Case (b), choose a $D$-colinear retraction $s : \wG \to \z$, and let 
$\theta= \op{id}\ot s: A \ot \wG \to A \ot \z$. 

Using these $\theta$, define $\D$-super-comodule superalgebra maps 
$\omega_{\theta} : \C \to (A \ot \zz)\, \square_D\, \D$
as in \eqref{eq:omega}
for all $A= \O_G(\pi^{-1}(U))$, where $U \subset |G/H|$ are affine open. 
Then the maps are seen to be compatible in the same sense as in Remark \ref{rem:compatible} (2),
with respect to all pairs $U \supset U'$ in $|G/H|$. It follows that all superalgebra
maps $\O_{\G}(\pi^{-1}(U))^{co \D}\to \O_{\G}(\pi^{-1}(U'))^{co \D}$
restricted from the restriction maps $\O_{\G}(\pi^{-1}(U))\to \O_{\G}(\pi^{-1}(U'))$
may be identified with those maps $A\, \square_D\, \zz\to A'\, \square_D\, \zz$ which
arise from the restriction maps $A = \O_G(\pi^{-1}(U)) \to \O_G(\pi^{-1}(U'))=A'$. 
In view of \eqref{eq:B} this proves the desired result. 
\end{proof}

\begin{proof}[Proof in Case (c)]
Assume (c). Then $\C = C \rcosmash \wwG$. One sees from Lemma \ref{lemdef:graded} 
(see Condition (iv)) that $\H$ as well is graded, so that $\D = D \rcosmash \wwH$. 

Let $U \subset |G/H|$ be non-empty affine open.
Set $A=\O_G(\pi^{-1}(U))$, and let $\iota : C \to A$ be the (right $D$-comodule) algebra map 
associated with $G \supset \pi^{-1}(U)$. Compose the now canonical isomorphism
\[ \psi =\gr \psi : \C \os{\simeq}{\longrightarrow} C \ot \wwG \, (=C \rcosmash \wwG)  \]
as in \eqref{eq:psi} with $\iota \ot \op{id} : C \ot \wwG \to A \ot \wwG$. 
The resulting
$\C \to A \ot \wwG$ is seen to be a right $\D$-super-comodule superalgebra map. 
Moreover, it gives rise to a right $\H$-equivariant embedding $\S(A \ot \wwG) \to \G$ of superschemes
onto the open subset $\pi^{-1}(U)$ of $|\G|\, (=|G|)$; see the proof of
Proposition \ref{prop:open_embed}. Here, the $\D$-super-comodule structure on 
$A \ot \wwG$ is such that $\wwH$ co-acts naturally on the tensor factor $\wwG$, and $D$ co-acts
co-diagonally on the tensor product. Therefore, the $\wwH$-coinvariants in
$A \ot \wwG$ are given by
\[ (A \ot \wwG)^{co\wwH}= A \ot \wedge(\z). \]
This last is the tensor product of two right $D$-comodules. Its $D$-coinvariants 
coincide with the $\D$-coinvariants in the original $A \ot \wwG$, and are given by
\[ (A \ot \wedge(\z))^{coD}=A\, \square_D\, \zz. \]
We thus have $\O_{\G/\H}(U)=\O_{\G}(\pi^{-1}(U))^{co\D} \simeq A\, \square_D\, \zz=
\wedge_B(A\, \square_D\, \z)$; see \eqref{eq:B}.
Since the isomorphism is natural in $U$, the desired result follows. 
\end{proof}

\begin{rem}\label{rem:known_for_split}
(1)\ The super-Grassmanians $\op{Gr}(s|r, m|n)$,  super-analogues of Grassmanians, are presented in the form 
$\G/\H$ as in \cite[Section 6]{MZ1};  $\op{Gr}(s|r, m|n)$ is a smooth algebraic superscheme which 
has the product $\op{Gr}(s, m)\times \op{Gr}(r, n)$ of Grassmanians as its associated scheme. 
It was early proved by Manin \cite[Chapter 4, Section 3, 16.~Example, p.200]{Manin} that the super-Grassmanian $\op{Gr}(1|1, 2|2)$ 
over the field of complex numbers is not split.

(2)\ E. G. Vishnyakova \cite{V1}, \cite{V2} studies the splitting property of quotients $\G/\H$ in the
analytic situation for complex super Lie groups. 
Theorem 2 of \cite{V1} proves
our Proposition \ref{prop:split} in Case (c) in the analytic situation. 
Example 3 of \cite{V2} tells us that the super-Grassmanian $\op{Gr}(s|r, m|n)$, constructed as a complex super-manifold, is not split  
if and only if $0<s<m$ and $0<r<n$. 
The ``if" holds as well for our algebraic $\op{Gr}(s|r, m|n)$ (over the field of complex numbers), since one can prove the following:
(i)~The analytic $\op{Gr}(s|r, m|n)$ is the analytification of ours;
(ii)~If a smooth locally-algebraic superscheme is split, then its analytification is, as well.
\end{rem}


\subsection{Brundan's work}\label{subsec:Brundan} 
Let us look closely at Brundan's paper \cite{B} cited in the Introduction. 
Retain the notation as above; see the beginning of Section \ref{subsec:open_embed}. 
Brundan assumed six properties, (Q1)--(Q6), which $\G \supset \H$ is expected to have. 
The first three (Q1)--(Q3) assume essentially that there exists a quotient superscheme $\G/\H$ such that
it is Noetherian and the quotient morphism $\G \to \G/\H$ is affine and faithfully flat. (Q4) assumes
that the scheme associated with $\G/\H$ is the quotient scheme $G/H$ for the associated affine algebraic groups.
The present paper as well as \cite{MZ1} have proved that these are all true in general and, moreover, 
that $\G/\H$ represents the faisceau $\G\tilde{/}\H$.
(Q5) assumes essentially that
$\G/\H$ is locally split, that is, split, restricted to some open neighborhood of every point. 
We have first proved that this is true in general; see Remark \ref{rem:about_mainthm}. 
The last (Q6) assumes that scheme $G/H$ is
projective. But it does happen that $G/H$ does not satisfy the assumption, as is remarked in \cite{B}. 

Brundan \cite[Section 2]{B} proved some useful, general results on $\G/\H$, assuming (Q1)--(Q6), and applied them to the special algebraic
supergroup $\G=Q(n)$ and its parabolic sub-supergroups $\H=P_{\gamma}$, proving that they satisfy the assumptions.
He thereby obtained beautiful results on modular representations of $Q(n)$. 

The general results above now hold for $\G/\H$ in general, 
only assuming that $G/H$ is projective, in case the result requires (Q6). 
Therefore, they can apply to investigate representations of a wider class of affine algebraic supergroups. 

\bigskip

\emph{Note added in revision.}\ 
A referee kindly suggested to the authors to add the articles \cite{BCF} \cite{Kostant} and 
\cite{V} to the References, in which the quotient problem had been considered before for differential and complex-analytic 
supergroups. We added also our subsequent preprint \cite{HMT} joint with M. Hoshi, which
gives a new, Hopf-algebraic construction of the quotients $\G/\H$ in the analytic situation when the base field is a complete field
of characteristic $\ne 2$. In fact, there is used analogous argument of proving Theorem \ref{mainthm} above. 
In view of Remark \ref{rem:known_for_split} (2), the referee posed to the authors a question which essentially asks
whether the complex-analytification functor is compatible with constructing quotients $\X/\G$,
which generalize the supergroup quotients $\G/\H$ discussed by the present paper. Our answer is positive under some appropriate
assumptions that include freeness of the action; details will appear in a forthcoming paper. 


\section*{Acknowledgments}
The first-named author was supported by
JSPS Grant-in-Aid for Scientific Research (C) 17K05189. 
The authors thank Alexandr Zubkov for his helpful comments on an earlier version of
this paper. Main results were announced at the 51st Symposium on Ring Theory and Representation Theory
held on September 19--22, 2018 at Okayama University of Science, Okayama, Japan. 



\begin{thebibliography}{99}

\bibitem{BCF}
L.~Balduzzi,\ C.~Carmeli,\ R.~Fioresi,\ \emph{Quotients in supergeometry};\ 
$\mathsf{arXiv}$:0805.3270.

\bibitem{B}
J.~Brundan,\
\emph{Modular representations of the supergroup $Q (n)$, II},\
Pacific J. Math. \textbf{224} (2006), no.~1, 65--90.



\bibitem{DG} M.~Demazure,\ P.~Gabriel,\ 
\emph{Groupes alg\'{e}briques,\ Tome I},\ 
Masson~$\And$~Cie,\ Paris; North-Holland, Amsterdam, 1970.
 

\bibitem{D}
Y.~Doi,\ 
\emph{Hopf extensions of algebras and Maschke type theorems},\
Israel J. Math. {\bf 72} (1990),\ no.~1--2,\ 99--108.  

\bibitem{HMT} 
M.~Hoshi,\ A.~Masuoka,\ Y.~Takahashi,\ 
\emph{Hopf-algebraic techniques applied to super Lie groups over a complete field};\
preprint $\mathsf{arXiv}$:1706.02839v4.

\bibitem{J}
J.~C.~Jantzen,\
\emph{Representations of Algebraic Groups},\
Pure and Applied Mathematics, Vol.~131, Academic Press, New York, 1987. 

\bibitem{Kostant}  
B.~Kostant,\
\emph{Graded manifolds, graded Lie theory, and prequantization},\
Lecture Notes in Mathematics, Vol.~570, Springer-Verlag, Berlin/Heidelberg/New York, 1977, pp.177--306.


\bibitem{Manin} 
Yu.~I.~Manin,\
\emph{Gauge Field Theory and Complex Geometry},\ Second edition,
Grundlehren der Mathematischen Wissenschaften,\ Vol.~289, 
Springer-Verlag,\ Berlin,\ 1997.

\bibitem{M1} 
A.~Masuoka,\
\emph{The fundamental correspondences in super affine groups and
super formal groups},\
J.\ Pure\ Appl.\ Algebra \textbf{202} (2005),\ no.~1--3,\ 284--312.
%
\bibitem{M2} A.~Masuoka,\ 
{\em Harish-Chandra pairs for algebraic affine supergroup schemes over an arbitrary field},\ 
Transform. Groups \textbf{17} (2012),\ no.~4,\ 1085--1121.

\bibitem{M3} 
A.~Masuoka,\
\emph{Hopf algebraic techniques applied to super algebraic groups},\ 
Proceedings of
Algebra Symposium (Hiroshima, 2013), pp. 48--66, Math. Soc. Japan, 2013;
available at 
$\mathsf{arXiv}$:~1311.1261v2.

\bibitem{MS} 
A.~Masuoka,\ T.~Shibata,\ 
\emph{On functor points of affine supergroups},\
J.\ Algebra \textbf{503} (2018), 534--572. 

\bibitem{MZ1} A.~Masuoka,\ A.~N.~Zubkov,\ 
\emph{Quotient sheaves of algebraic supergroups are superschemes},\ 
J.\ Algebra \textbf{348} (2011), 135--170.

\bibitem{MZ2} A.~Masuoka,\ A.~N.~Zubkov,\ 
\emph{Solvability and nilpotency for algebraic supergroups},\ 
J.\ Pure\ Appl.\ Algebra\ \textbf{221} (2017), no.~2,\ 339--365.


\bibitem{MZ3} A.~Masuoka,\ A.~N.~Zubkov,\ 
\emph{The structure of general algebraic supergroups} (tentative), in preparation. 


\bibitem{Mo} S.~Montgomery,\
\emph{Hopf Algebras and Their Actions on Rings},\
CBMS Conf.\ Series in Math., Vol.~82, Amer.\ Math.\ Soc., Providence, RI, 1993.

\bibitem{O} U.~Oberst,\ 
\emph{Affine Quotientenschemata nach affine, algebraischen Gruppen und induzierte Darstellungen},\
J.\ Algebra \textbf{44} (1977), no.~2,\ 503--538. 


\bibitem{Schmitt}
T.~Schmitt, 
\emph{Regular sequences in $\mathbb{Z}_2$-graded commutative algebra},\
J.\ Algebra \textbf{124} (1989), no.~1,\ 60--118.


\bibitem{S} H.-J.~Schneider,\  
\emph{Principal homogeneous spaces for arbitrary Hopf algebras}, 
Israel J. Math. \textbf{72} (1990), no.~1--2,\ 167--195. 

\bibitem{Sw} 
M.~E.~Sweedler,\ 
\emph{Hopf Algebras},\
W. A. Benjamin, Inc., New York,~1969.

\bibitem{T1} M.~Takeuchi,\
\emph{Formal schemes over fields},\
Comm. Algebra {\bf 5} (1977),\ no.~14,\ 141438--1528.

\bibitem{T2} M.~Takeuchi,\
\emph{Relative Hopf modules---equivalences and freeness criteria},\
J.\ Algebra \textbf{60} (1979),\ no.~2,\  452--471. 

\bibitem{V}
E.~G.~Vishnyakova,\
\emph{On the structure of complex homogeneous supermanifolds}; 
$\mathsf{arXiv}$:0811.2581.

\bibitem{V1} E.~G.~Vishnyakova,\ 
{\em On complex Lie supergroups and split homogeneous supermanifolds}, 
Transform. Groups \textbf{16} (2011),\ no.~1,\ 265--285. 

\bibitem{V2} E.~G.~Vishnyakova,\ 
\emph{The splitting problem for complex
homogeneous supermanifolds},\
J.\ Lie Theory \textbf{25} (2015),\ no.~2,\ 459--476.

\bibitem{Z} 
A.~N.~Zubkov,\
\emph{Affine quotients of supergroups},\
Transform. Groups \textbf{14} (2009),\ no.~3,\ 713--745. 

\end{thebibliography}
\end{document}